\documentclass[english, 10pt,reqno]{amsart}
\usepackage{geometry}  
\usepackage{amssymb,amsmath,amsthm,amsfonts,color}
\usepackage{mathrsfs,dsfont, comment,mathscinet}
\usepackage{graphicx}
\usepackage{epstopdf}
\usepackage{mathtools}
\usepackage{babel}
\usepackage{enumerate,esint}

\usepackage{indentfirst}
\usepackage{bm}
\usepackage{picinpar}
\usepackage{caption}
\usepackage{subcaption}
\usepackage{soul}

\usepackage[colorlinks=true, pdfstartview=FitV, linkcolor=blue, citecolor=blue, urlcolor=blue]{hyperref}

\usepackage[toc,page]{appendix}

\usepackage{etoolbox}
\usepackage{authblk}
\usepackage{amsaddr}

\allowdisplaybreaks %This command allows 

\mathtoolsset{showonlyrefs} %This command will only number the equation which is referred later.

\makeatletter

\usepackage{fancyhdr}
 
\makeatletter
\patchcmd{\@maketitle}
  {\ifx\@empty\@dedicatory}
  {\ifx\@empty\@date \else {\vskip3ex \centering\footnotesize\@date\par\vskip1ex}\fi
   \ifx\@empty\@dedicatory}
  {}{}
\patchcmd{\@maketitle}
  {\ifx\@empty\@date\else \@footnotetext{\@setdate}\fi}
  {}{}{}
\makeatother
\begingroup
\newtheorem{theorem}{Theorem}[section]
\newtheorem{lemma}[theorem]{Lemma}
\newtheorem{proposition}[theorem]{Proposition}
\newtheorem{corollary}[theorem]{Corollary}

\endgroup

\theoremstyle{definition}
\begingroup
\newtheorem{define}[theorem]{Definition}
\newtheorem{remark}[theorem]{Remark}

\endgroup
\newcommand\ba[1]{\begin{align}\label{#1}}
\newcommand\ea{\end{align}}
\newcommand\bas{\begin{align*}}
\newcommand\eas{\end{align*}}

\newcommand\ee{\end{equation}}
\newcommand\be{\begin{equation}}
\newcommand\ees{\end{equation*}}
\newcommand\bes{\begin{equation*}}
\mathchardef\emptyset="001F

\newcommand{\e}{\varepsilon}
\newcommand{\ep}{\varepsilon}

\newcommand{\R}{{\mathbb R}}

\newcommand{\wtos}{\mathrel{\mathop{\rightharpoonup}\limits^*}}

\newcommand{\N}{{\mathbb{N}}}

\newcommand\norm[1]{\left\|#1\right\|}

\newcommand{\abs}[1]{\left\lvert#1\right\rvert} 
\newcommand{\fsp}[1]{\left(#1\right)} 
\newcommand{\fmp}[1]{\left[#1\right]}
\newcommand{\flp}[1]{\left\{#1\right\}}

\newcommand{\limn}{\lim_{n\rightarrow\infty}}

\newcommand{\seqn}[1]{\left\{#1\right\}}

\newcommand\nn{\nonumber}

\definecolor{CMUred}{RGB}{153,0,0}
\definecolor{CMUgreen}{RGB}{0,135,81}
\definecolor{CMUblue}{RGB}{0,51,127}
\newcommand{\argmin}{{\operatorname{arg\,min}}}

\newcommand{\ta}{{\alpha}}
\newcommand{\bs}{{\bar s}}
\newcommand\F{\mathcal{F}}

\def\argmin{\mathop{\rm arg\, min}}

\numberwithin{equation}{section}

\newcommand{\normmm}[1]{{\left\vert\kern-0.25ex\left\vert\kern-0.25ex\left\vert #1 
    \right\vert\kern-0.25ex\right\vert\kern-0.25ex\right\vert}}
\newcommand{\AAA}{\color{black}}
\newcommand{\RRR}{\color{red}}

\title{One dimensional fractional order $TGV$: Gamma-convergence and bilevel training scheme}

\author{Elisa Davoli }
 \address[Elisa Davoli ]{Faculty of Mathematics, University of Vienna\\ Oskar-Morgenstern-Platz 1, A-1090 Vienna, Austria}
 \email[Elisa Davoli ] {elisa.davoli@univie.ac.at}
\author[P. Liu] {Pan Liu}
 \address[Pan Liu]{Cambridge Image Analysis, Department of Applied Mathematics and Theoretical Physics,\\ University of Cambridge, Wilberforce Road, Cambridge CB3 0WA, UK}
 \email[P. Liu] {panliu.0923@maths.cam.ac.uk}

\date{}                                           % Activate to display a given date or no date

\begin{document}

%\linenumbers
    %\lipsum

\begin{abstract}
New fractional $r$-order seminorms, $TGV^r$, $r\in \R$, $r\geq 1$, are proposed in the one-dimensional (1D) setting, as a generalization of the integer order $TGV^k$-seminorms, $k\in\mathbb{N}$. The fractional $r$-order $TGV^r$-seminorms are shown to be intermediate between the integer order $TGV^k$-seminorms. A bilevel training scheme is proposed, where under a box constraint a simultaneous optimization with respect to parameters and order of derivation is performed. Existence of solutions to the bilevel training scheme is proved by $\Gamma$--convergence. Finally, the numerical landscape of the cost function associated to the bilevel training scheme is discussed for two numerical examples.
\end{abstract}
\maketitle

\thispagestyle{empty}%this command remove the page number at the title page

%\section{Problems to fix}
%We are listing a series of questions below:
%\begin{enumerate}[1.]
%\item the Proposition \ref{compact_seminorm} should take first priority now. It is the key point to our new training scheme. I add a small idea of prove.
%\item the new scheme does not require of the optimizing order. i.e., we do not seek for $\alpha$ first and $s$ later, but more or less at the same time.
%
%
%\end{enumerate}
\section{Introduction}
In the last decades, Calculus of Variations and Partial Differential Equations (PDE) methods have proven to be very efficient in signal (1D) and image (2D) denoising problems. Signal (image) denoising consists, roughly speaking, in recovering a noise-free clean signal $u_c$ starting from a corrupted signal $u_\eta=u_c+\eta$, by filtering out the noise encoded by $\eta$.
 One of the most successful variational approach to signal (image) denoising (see, for example \cite{rudin1993segmentation, rudin1994total, rudin1992nonlinear}) relies on the ROF total-variational functional
\begin{align}\label{mumfordshahori}
ROF(u):=\norm{u-u_\eta}_{L^2(I)}^2+\alpha TV(u),
\end{align}
introduced in \cite{rudin1993segmentation}. Here $I=(0,1)$ represents the domain of a one-dimensional image (a signal), $\alpha\in\mathbb R^+$, and $TV(u)=\abs{u'}_{\mathcal{M}_b(I)}$, stands for the total mass of the measure $u'$ on $I$ (see \cite[Definition 1.4]{ambrosio.fusco.pallara}).\\ 
%By minimizing the functional \eqref{mumfordshahori} one tries to find a ``piecewise continuous'' approximation of $u_\eta$. The existence of such minimizers can be proved by using compactness and lower semicontinuity theorems in $SBV(I)$ (see \cite{ambrosiocompactness,  ambrosio1989variational, ambrosio1990existence,ambrosio2000functions}). Here, and in what follows, $J_u$ stands for the jump set of $u$.\\\\
%

An important role in determining the reconstruction properties of the ROF functional is played by the parameter $\alpha$. Indeed, if $\alpha$ is too large, then the total variation of $u$ is too penalized and the image turns out to be over-smoothed, with a resulting loss of information on the internal edges of the picture. 
Conversely, if $\alpha$ is too small then the noise remains un-removed. The choice of the ``best" parameter $\alpha$ then becomes an important task.\\

 In \cite{reyes2015structure} the authors proposed a training scheme $(\mathcal B)$ relying on a bilevel learning optimization defined in machine learning, namely on a semi-supervised training scheme that optimally adapts itself to the given ``perfect data" (see \cite{chen2013revisiting, chen2014insights, domke2012generic, domke2013learning, tappen2007utilizing,tappen2007learning}). This training scheme searches for the optimal $\alpha$ so that the recovered image $u_\alpha$, obtained as a minimizer of \eqref{mumfordshahori}, optimizes the $L^2$-distance from the clean image ${u_c}$. An implementation of $(\mathcal B)$ equipped with total variation is the following: 
\begin{flalign}
\text{Level 1. }&\,\,\,\,\,\,\,\,\,\,\,\,\,\,\,\,\,\,\,\,\,\,\,\,\,\,\,\,\,\,\,\,\,\,\,\,\,\alpha_m\in\argmin\flp{\norm{u_\alpha-u_c}_{L^2(I)}^2:\,\,\alpha>0}&\\
\text{Level 2. }&\,\,\,\,\,\,\,\,\,\,\,\,\,\,\,\,\,\,\,\,\,\,\,\,\,\,\,\,\,\,\,\,\,\,\,\,\,u_{\alpha}:=\argmin\flp{\norm{u-u_\eta}_{L^2(I)}^2+\alpha TV(u):\,\,u\in SBV(I)},\label{result_level_bi}&
\end{flalign}
where $SBV(I)$ denotes the set of special functions of bounded variation in $I$ (see, e.g. \cite[Chapter 4]{ambrosio.fusco.pallara}).\\

It is well known that the ROF model in \eqref{mumfordshahori} suffers drawbacks like the staircasing effect, and the training scheme $(\mathcal B)$ inherits that feature, namely the optimized reconstruction function $u_{\alpha_m}$ also exhibits the staircasing effect. One approach to counteract this problem is to insert higher-order derivatives in the regularizer (see \cite{burger2015infimal,  chan2000high, maso2009higher,papafitsoros2013study}). Two of the most successful image-reconstruction functionals among those involving mixed first and higher order terms are the infimal-convolution total variation ($ICTV$) \cite{burger2015infimal} and the total generalized variation ($TGV$) models \cite{papafitsoros2013study}. Note that they coincide with each other in the one-dimensional setting.\\\\
For $I:=(0,1)\subset \R$, $u\in BV(I)$, $k\in \mathbb N$, and $\alpha=(\alpha_0,\dots,\alpha_k)\in \R^{k+1}_{+}$, the $TGV^k_{\alpha}$ regularizer (see \cite{valkonen2013total}) is defined as
\be
\abs{u}_{TGV^1_{\alpha}(I)}:=\alpha_0 TV(u),
\ee
and
\begin{align*}
\label{inter_TGV}
\nonumber\abs{u}_{TGV_{\ta}^{k+1}(I)}:=\inf&\flp{\alpha_0\abs{u'- v_0}_{\mathcal{M}_b(I)}+\alpha_1\abs{v_0'-v_1}_{{\mathcal{M}_b(I)}}+\right.\\
&\nonumber \qquad \left.\dots+\alpha_{k-1}\abs{v_{k-2}'-v_{k-1}}_{{\mathcal{M}_b(I)}}+\alpha_{k}\abs{v'_{k-1}}_{{\mathcal{M}_b(I)}}:\right.\\
&\qquad\left. v_i\in BV(I)\text{ for }0\leq i\leq k-1}.
\end{align*}
For instance, for $k=1$, the $TGV^2_{\alpha_0,\alpha_1}$ regularizer reads as
\bes
\abs{u}_{TGV_{\alpha_0,\alpha_1}^{2}(I)}:=\inf\flp{\alpha_0\abs{u'- v_0}_{\mathcal{M}_b(I)}+\alpha_1\abs{v'_0}_{{\mathcal{M}_b(I)}},\,\, v_0\in BV(I)}.
\ees

Substituting $TGV^2_{\alpha_0,\alpha_1}$ into \eqref{result_level_bi} provides a bilevel training scheme with $TGV$ image-reconstruction model. We recall that large values of $\alpha_1$ will yield regularized solutions that are close to $TV$-regularized reconstructions, and large values of $\alpha_0$ will result in $TV^2$-type solutions (see, e.g., \cite{papafitsoros2014combined}). The best choice of parameters $\alpha_0$ and $\alpha_1$ is determined by an adaptation of the training scheme $(\mathcal B)$ above (see \cite{2015arXiv150807243D} for a detailed study). \\\\
In the existing literature a regularizer is fixed a priori, and the biggest effort is concentrated on studying how to identify the best parameters. In the case of the $TGV^k$-model, this amounts to set manually the value of $k$ first, and then determine the optimal $\alpha_m$ in \eqref{result_level_bi}. However, there is no evidence suggesting that $TGV^2$ will always perform better than $TV$. In addition, the higher order seminorms $TGV^{k}$, $k\geq 2$, have rarely been analyzed, and hence their performance is largely unknown. Numerical simulations show that for different images (signals in 1D), different orders of $TGV^{k}$ might give different results. The main focus of this paper is exactly to investigate how to optimally tune both the weight $\alpha$ and the order $k$ of the $TGV^{k}_{\ta}$-seminorm, in order to achieve the best reconstructed image.\\\\
Our result is threefold. First, we develop a bilevel training scheme, not only for parameter training, but also for determining the optimal order $k$ of the regularizer $TGV^{k}$ for image reconstruction. A straightforward modification of $(\mathcal B)$ would be to just insert the order of the regularizer inside the learning level 2 in \eqref{result_level_bi}. Namely,
\begin{flalign}
\text{Level 1. }&\,\,\,\,\,\,\,\,\,\,\,\,\,\,\,\,\,\,\,\,\,\,\,\,(\tilde \alpha,\tilde k):=\argmin\flp{\norm{u_{\alpha,k}-u_c}_{L^2(I)}^2:\,\,\ta>0,\,\,k\in \mathbb N}&\\
\text{Level 2. }&\,\,\,\,\,\,\,\,\,\,\,\,\,\,\,\,\,\,\,\,\,\,\,\,u_{\ta,k}:=\argmin\flp{\norm{u-u_\eta}_{L^2(I)}^2+\abs{u}_{TGV_{\ta}^{k}(I)}:\,\,u\in SBV(I)}.&
\end{flalign}
% \begin{enumerate}[Level 1.]
%\item
%\be
%(\tilde \alpha,\tilde k):=\argmin_{\ta>0,k\in \mathbb N}\frac12\int_I \abs{u_{\ta,k}-u_c}^2dx,
%\ee
%\item
%\be\label{fr_t_level_bi}
%u_{\ta,k}:=\argmin_{u\in SBV(Q)}\flp{\norm{u-u_\eta}_{L^2(I)}^2+\abs{u}_{TGV_{\ta}^{k+1}(I)}}.
%\ee
%\end{enumerate}

Often, in order to show the existence of a solution of the training scheme and also for the numerical realization of the model, a \emph{box constraint} 
\be\label{box_const_intro}
(\alpha, k)\in [\mathrm{P}, \mathrm{1/P}]^{k+1}\times [1,\mathrm{1/P}],
\ee
where $P\in(0,1)$ small is a fixed real number, needs to be imposed (see, e.g. \cite{bergounioux1998optimal, de2013image}). However, such constraint makes the above training scheme less interesting. To be precise, restricting the analysis to the case in which $k\in\mathbb N$ is an integer, the box constraint \eqref{box_const_intro} would only allow $k$ to take finitely many values, and hence the optimal order $\tilde k$ of the regularizer would simply be determined by performing scheme $(\mathcal B)$ finitely many times, at each time with different values of $k$. In addition, finer texture effects, for which an ``intermediate" reconstruction between the one provided by $TGV^k$ and $TGV^{k+1}$ for some $k\in\mathbb{N}$ would be needed, might be neglected in the optimization procedure.\\

Therefore, a main challenge in the setup of such a training scheme is to give a meaningful interpolation between the spaces $TGV^{k}$ and $TGV^{k+1}$, and hence to guarantee that the collection of such spaces itself exhibits certain compactness and lower semicontinuity properties. To this purpose, we modify the definition of the $TGV^{k}$ functionals by incorporating the theory of fractional Sobolev spaces, and we introduce the notion of fractional order $TGV^{k+s}$ spaces (see Definition \ref{TGV_fractional}), where $k\in\mathbb N$, and $0<s<1$. For $k=1$, our definition reads as follows.
\begin{multline*}
\abs{u}_{TGV_{\alpha}^{1+s}(I)}:=\inf\flp{\alpha_0\abs{u'- sv_0}_{\mathcal{M}_b(I)}+\alpha_1{s(1-s)}\abs{v_0}_{W^{s,1+s(1-s)}(I)}\right.\\
 \left.+\alpha_0 s(1-s)\Big|\int_I v_0(x)\,dx\Big|:\, v_0\in W^{s,1+s(1-s)}(I)}.
\end{multline*}
In the expression above, $W^{s,1+s(1-s)}(I)$ is the \emph{fractional Sobolev space} of order $s$ and integrability $1+s(1-s)$. For every $k\in\N$ and $s\in[0,1]$ we additionally introduce the sets
\be
BGV_{\alpha}^{k+s}(I):=\flp{u\in L^1(I):\, \abs{u}_{TGV_{\alpha}^{k+s}(I)}<+\infty},
\ee
 namely the classes of functions with bounded generalized total-variation seminorm.\\

In our first main result (see Theorem \ref{intermediate}) we show that the $TGV^{1+s}$ seminorm is indeed intermediate between $TGV^1$ and $TGV^2$, i.e., we prove that, up to subsequences, 
\be\label{intro_equiv_equiv}
\lim_{s\nearrow 1}\abs{u}_{TGV^{1+s}_{\alpha}(I)}= \abs{u}_{TGV^2_{\alpha}(I)}\text{ and }\lim_{s\searrow 0}\abs{u}_{TGV^{1+s}_{\alpha}(I)}=\alpha_0\abs{u'}_{\mathcal{M}_b(I)}.
\ee
Equation \eqref{intro_equiv_equiv} shows that, for $s\nearrow 1$, the behavior of the $TGV^{1+s}$-seminorm is close to the one of the standard $TGV^2$-seminorm, whereas for $s\searrow 0$ it approaches the $TV$ functional. We additionally prove (see Corollary \ref{cor:higher-order}) that analogous results hold for higher order $TGV^{k+s}$-seminorms.
We point out that working with such interpolation spaces has many advantages. Indeed, $TGV^{k+s}$ is expected to inherit the properties of fractional order derivatives, which have shown to be able to reduce the staircasing and contrast effects in noise-removal problems (see, e.g. \cite{chen2013fractional}).\\

Our second and third main results (Theorems \ref{thm:new-Gamma} and \ref{main_result}) concern the following improved training scheme $(\mathcal R)$, which, under the box constraint \eqref{box_const_intro}, simultaneously optimizes both the parameter $\alpha$ and the order $r$ of derivation:
\begin{flalign}
\text{Level 1. }&\,\,\,\,\,\,\,\,\,\,\,\,\,\,(\tilde \alpha,\tilde r):=\argmin\flp{\norm{u_{\alpha,r}-u_c}_{L^2(I)}^2,\,\,(\alpha, r)\in [\mathrm{P}, \mathrm{1/P}]^{{\lfloor r\rfloor}+1}\times [1,\mathrm{1/P}]},\label{frac_para_bi_intro}&\\
\text{Level 2. }&\,\,\,\,\,\,\,\,\,\,\,\,\,\,u_{\ta,r}:=\argmin\flp{\norm{u-u_\eta}_{L^2(I)}^2+TGV_{\alpha}^{r}(u):\,\, u\in BGV_{\alpha}^{r}(I)}.&
\end{flalign}
In the definition above, ${\lfloor r\rfloor}$ denotes the largest integer strictly smaller than or equal to $r$.

We first show in Theorem \ref{thm:new-Gamma} that the fractional order $TGV^r_{\alpha}$ functionals 
\be
 \F^{r}_{\alpha}(u):=\norm{u-u_\eta}_{L^2(I)}^2+TGV_{\alpha}^{r}(u)\quad\text{for every }u\in BGV_{\alpha}^{r}(I)
 \ee
 are continuous, in the sense of $\Gamma$-convergence in the weak* topology of $BV(I)$ (see \cite{braides} and \cite{dalmaso}), with respect to the parameters $\alpha$ and the order $r$. Secondly, in Theorem \ref{main_result} we exploit this $\Gamma$-convergence result, to prove existence of solutions to our training scheme $(\mathcal R)$. Note that, according to the given noisy image $u_\eta$ and noise-free image $u_c$, the Level 1 in our training scheme $(\mathcal R)$ provides simultaneously an optimal regularizer $TGV^{\tilde r}$ and a corresponding optimal parameter $\tilde \alpha\in [\mathrm{P}, \mathrm{1/P}]^{{\lfloor \tilde r\rfloor}+1}$. We point out that, in general, the optimal order of derivation $\tilde{r}$ might, or might not, be an integer. In other words, the fractional $TGV^r$ seminorms are not intended as an improvement but rather as an extension of the integer order $TGV^k$ seminorms, which for some classes of signals might provide optimal reconstruction and be selected by the bilevel training scheme.\\

Although this paper mainly focuses on a theoretical analysis of $TGV^r$ and on showing the existence of optimal results for the training scheme $(\mathcal R)$, in Section \ref{num_sim} some preliminary numerical examples are discussed (see Figures \ref{fig:fig1}--\ref{fig:fig2}). We stress that a complete description of the optimality conditions and a reliable numerical scheme for identifying the optimal solution of the training scheme \eqref{frac_para_bi_intro} are beyond the scope of this work, and are still a challenging open problem. We refer to \cite{de2017numerical,2016arXiv160201278B} for some preliminary results in this direction. The two-dimensional setting of fractional order $TGV^r$ and $ICTV^r$ seminorms, as well as more extensive numerical analysis and examples for different type of images (with large flat areas, fine details, etc.), will be the subject of the follow-up work \cite{EDPL_TGV_tgv}.\\\\ 
Our paper is organized as follows. In Section \ref{prelimy} we review the definitions and some basic properties of fractional Sobolev spaces. In Section \ref{sec_frac_TGV} we introduce the fractional order $TGV^{r}$ seminorms, we study their main properties, and prove that they are intermediate between integer-order seminorms (see Theorem \ref{intermediate}). In Section \ref{sec:Gamma} we characterize the asymptotic behavior of the functionals $\F^{r}_{\alpha}$ with respect to parameters and order of derivations (see Theorem \ref{thm:new-Gamma}). In Section \ref{BLSFRAC} we introduce our training scheme $(\mathcal R)$. In particular, in Theorem \ref{main_result} we show that  $(\mathcal R)$ admits a solution under the box constraint \eqref{box_const_intro}.  Lastly, in Section \ref{num_sim} some examples and insights are provided.

\section{The theory of fractional Sobolev Spaces}\label{prelimy}
In what follows we will assume that $I=(0,1)$. We first recall a few results from the theory of fractional Sobolev spaces. We refer to \cite{dinezza.palatucci.valdinoci} for an introduction to the main results, and to \cite{adams,landkof,leoni,mazya} and the references therein for a comprehensive treatment of the topic. 
\begin{define}[Fractional Sobolev spaces]
For $0<s<1$, $1\leq p<+\infty$, and $u\in L^p(I)$, we define the \emph{Gagliardo seminorm} of $u$ by 
\be\label{tv_k_method}
\abs{u}_{W^{s,p}(I)}:=\fsp{\int_{I}\int_{I}\frac{\abs{u(x)-u(y)}^p}{|x-y|^{1+sp}}\,dx\,dy}^{\frac1p}.
\ee

%and total deformation by
%\be\label{td_k_method}
%\abs{v}_{TD^s(I)}:=\sum_{i,j=1}^N\fsp{\int_0^\infty\norm{\frac12\fsp{\Delta_i^h v+\Delta_j^h v}}_{L^1(I)}\frac{1}{h^{1+s}}dh},
%\ee
%where
%\be
%\Delta_i^h v(x):=v(x+h e_i)-v(x)
%\ee
%in which $e_i$ is the $i$th vector of the canonical basis in $\rn$.\\\\
%
We say that $u\in W^{s,p}(I)$ if
\be
\norm{u}_{W^{s,p}(I)}:=\norm{u}_{L^p(I)}+\abs{u}_{W^{s,p}(I)}<+\infty.
\ee
%and $u\in BD^s$ if
%\be
%\norm{u}_{L^1(I)}+TD^s(u)<+\infty.
%\ee
\end{define}
%Note that the fractional derivative defined in \eqref{tv_k_method} is the $K$-method introduced in \cite{adams2003sobolev}, page 209, Section 7.9.
%The following inequality is a special case of \cite[Theorem 1]{bourgain.brezis.mironescu} and \cite[Theorem 1]{mazia.shaposhnikova}.
%\begin{theorem}[Poincar\'e Inequality in fractional Sobolev Spaces]\label{Poincare_TV}
%There exists a constant $C=C(I)>0$ such that 
%\be
%\norm{u-\int_{I}u(x)\,dx}_{L^{\tfrac{1}{1-s}}(I)}\leq Cs(1-s)\abs{u}_{W^{s,1}(I)}.
%\ee
%\end{theorem}
The following embedding results hold true (\cite[Theorems 6.7, 6.10, and 8.2, and Corollary 7.2]{dinezza.palatucci.valdinoci}).
\begin{theorem}[Sobolev Embeddings - 1]
\label{embedding} Let $s\in(0,1)$ be given.
\begin{enumerate}[$1$.]
\item Let $p<\frac{1}{s}$. Then there exists a positive constant $C=C(p,s)$ such that for every $u\in W^{s,p}(I)$ there holds
\be
\label{eq:embedding}\|u\|_{L^q(I)}\leq C\|u\|_{W^{s,p}(I)}\ee
for every $q\in [1,\tfrac{p}{1-sp}]$. If $q<\tfrac{p}{1-sp}$, then the embedding of $W^{s,p}(I)$ into $L^q(I)$ is also compact.

\item Let $p=\frac{1}{s}$. Then the embedding in \eqref{eq:embedding} holds for every $q\in [1,+\infty).$

\item Let $p>\frac{1}{s}$. Then there exists a positive constant $C=C(p,s)$ such that for every $u\in W^{s,p}(I)$ we have
\bes
\|u\|_{\mathcal C^{0,\alpha}(I)}\leq C\|u\|_{W^{s,p}(I)},
\ees
with $\alpha:=\frac{sp-1}{p}$, where  
\be
\norm{u}_{\mathcal C^{0,\alpha}(I)}:=\|u\|_{L^{\infty}(I)}+\sup_{x\neq y\in I}\frac{|u(x)-u(y)|}{|x-y|^{\alpha}}.
\ee
\end{enumerate}
\end{theorem}
The additional embedding result below is proved in \cite[Corollary 19]{MR1108473}. 
\begin{theorem}[Sobolev Embeddings - 2]
\label{thm:embedding-pq}
Let $s\geq r$, $p\leq q$ and $s-1/p\geq r-1/q$, with $0<r\leq s<1$, and $1\leq p\leq q\leq +\infty$. Then 
\be
W^{s,p}(I)\subset W^{r,q}(I),
\ee
and 
\bes
\abs{u}_{W^{r,q}(I)}\leq \frac{36}{rs}\abs{u}_{W^{s,p}(I)}.
\ees
\end{theorem}
%\begin{remark}\label{non_blow_up}
%By Lemma 6.2 in \cite{dinezza.palatucci.valdinoci}, we have the constant $C(s,I)$ in \eqref{embedding} is such that $\limsup_{s\searrow 0}C(s,I)=1$ provided that $q=1$.
%\end{remark}
%\begin{theorem}[Section $2.7.1$, page 129, \cite{MR3024598}]\label{embedding_besov}
%Let $0<p_0\leq p_1\leq +\infty$ and $-\infty<s_1<s_0<+\infty$. Then 
%\be
%B^{s_0}_{p_0,p_0}(\R)\subset B^{s_1}_{p_1,p_1}(\R)\text{ if }s_0-\frac1{p_0}=s_1-\frac1{p_1}.
%\ee
%\end{theorem}

The next inequality is a special case of \cite[Theorem 1]{bourgain.brezis.mironescu2} and \cite[Theorem 1]{mazia.shaposhnikova}.
\begin{theorem}[Poincar\'e Inequality]\label{Poincare_TV}
Let $p\geq 1$, and let $sp<1$. There exists a constant $C>0$ such that 
\be
\norm{u-\int_{I}u(x)\,dx}^p_{L^{\tfrac{p}{1-sp}}(I)}\leq \frac{Cs(1-s)}{(1-sp)^{p-1}}\abs{u}^p_{W^{s,p}(I)}.
\ee
\end{theorem}

It is possible to construct a continuous extension operator from $W^{s,1}(I)$ to $W^{s,1}(\R)$ (see,e.g., \cite[Theorem 5.4]{dinezza.palatucci.valdinoci}).
\begin{theorem}[Extension Operator]
\label{extension}
Let $s\in (0,1)$, and let $1\leq p<+\infty$. Then $W^{s,p}(I)$ is continuously embedded in $W^{s,p}(\R)$, namely there exists a constant $C=C(p,s)$ such that for every $u\in W^{s,p}(I)$ there exists $\tilde{u}\in W^{s,p}(\R)$ satisfying $\tilde{u}|_{I}=u$ and
\be\norm{\tilde{u}}_{W^{s,p}(\R)}\leq C\norm{u}_{W^{s,p}(I)}.\ee
\end{theorem}
The next two theorems (\cite[Section $2.2.2$, Remark 3, and Section $2.11.2$]{MR3024598}) yield an identification between fractional Sobolev spaces and Besov spaces in $\R$, and guarantee the reflexivity of Besov spaces $B^{s}_{p,q}$ for $p,q$ finite.
\begin{theorem}[Identification with Besov spaces]\label{frac_sobo_to_bes}
If $1\leq p<+\infty$ and $s\in\R^+\setminus\mathbb N$, then 
\be
W^{s,p}(\R)=B^s_{p,p}(\R)
\ee
\end{theorem}

\begin{theorem}[Reflexivity of Besov spaces]\label{Besov_refl}
Let $-\infty <s<+\infty$, $1\leq p<+\infty$ and $0<q<+\infty$. Then
\be
(B^s_{p,q}(\R))'=B^{-s}_{p',q'}(\R),
\ee
where $(B^s_{p,q}(\R))'$ is the dual of the Besov space $B^s_{p,q}(\R)$, and where $p'$ and $q'$ are the conjugate exponent of $p$ and $q$, respectively.
\end{theorem}
In view of Theorems \ref{frac_sobo_to_bes} and \ref{Besov_refl} the following characterization holds true.
\begin{corollary}[Reflexivity of fractional Sobolev spaces]\label{sobo_frac_ref}
Let $1< p<+\infty$ and $s\in\R^+\setminus \mathbb N$. Then the fractional Sobolev space $W^{s,p}(\R)$ is reflexive.
\end{corollary}
%The next result (see \cite{MR3024598}[Section $2.7.1$]) characterizes embeddings between Besov spaces of different orders.

%\begin{theorem}\label{embedding_besov}
%Let $0<p_0\leq p_1\leq +\infty$ and $-\infty<s_1<s_0<+\infty$. Then 
%\be
%B^{s_0}_{p_0,p_0}(\R)\subset B^{s_1}_{p_1,p_1}(\R)\text{ if }s_0-\frac1{p_0}=s_1-\frac1{p_1}.
%\ee
%\end{theorem}

We conclude this section by recalling two theorems describing the limit behavior of the Gagliardo seminorm as $s\nearrow 1$ and $s\searrow 0$, respectively. The first result has been proved in \cite[Theorem 3 and Remark 1]{bourgain.brezis.mironescu}, and \cite[Theorem 1]{davila}.
\begin{theorem}[Asymptotic behavior as $s\nearrow 1$]\label{approx_frac_der1}
Let $u\in BV(I)$. Then
\begin{align*}
 \lim_{s\nearrow 1} (1-s)\abs{u}_{W^{s,1}(I)}=|u'|_{{\mathcal{M}_b(I)}}.
\end{align*} 
\end{theorem}
Similarly, the asymptotic behavior of the Gagliardo seminorm has been characterized as $s\searrow 0$ in \cite[Theorem 3]{mazia.shaposhnikova}.
\begin{theorem}[Asymptotic behavior as $s\searrow 0$]\label{approx_frac_der2}
Let $u\in\cup_{0<s<1}W^{s,1}(\R)$. Then,
\bes
\lim_{s\searrow 0} s\abs{u}_{W^{s,1}(\R)}=4\norm{u}_{L^1(\R)}.
\ees
\end{theorem}

\section{The Fractional order $TGV$ seminorms}\label{sec_frac_TGV}
Let $r\in(1,+\infty)\setminus \N$ be given. In this section we define the \emph{fractional $r$-order total generalized variation} ($TGV^r$) seminorms, and we prove some first properties.
\begin{define}[The $TGV^r$ seminorms]\label{TGV_fractional}
Let $0<s<1$, $k\in\mathbb N$ be such that $k+s=r$, and let $\alpha=(\alpha_0,\alpha_1,\alpha_2,\ldots, \alpha_{k})\in\R_+^{k+1}$. 
For every $u\in L^1(I)$, we define its fractional $TGV^{k+s}$ seminorm as follows.

\begin{enumerate}[C{a}se 1.] 
%\item
%for $k=0$
%\be
%\abs{u}_{TGV_{\alpha}^{s}(I)}:=\alpha_0{s(1-s)}\abs{u}_{W^{s,1+s(1-s)}(I)}+\alpha_0 s(1-s)\abs{\int_I u(x)\,dx}
%\ee
%\AAA
\item
for $k=1$
\begin{multline}
\abs{u}_{TGV_{\alpha}^{1+s}(I)}:=\inf\flp{\alpha_0\abs{u'- sv_0}_{\mathcal{M}_b(I)}+\alpha_1{s(1-s)}\abs{v_0}_{W^{s,1+s(1-s)}(I)}\right.\\
\left.+\alpha_0 s(1-s)\abs{\int_I v_0(x)\,dx}\AAA:\,v_0\in W^{s,1+s(1-s)}(I)}.
\end{multline}
\item
for $k>1$
\begin{align*}
\abs{u}_{TGV_{\alpha}^{k+s}(I)}:=\inf&\flp{\alpha_0\abs{u'- v_0}_{\mathcal{M}_b(I)}+\alpha_1\abs{v_0'-v_1}_{{\mathcal{M}_b(I)}}+\right.\\
&\left. \cdots+\alpha_{k-1}\abs{v_{k-2}'-sv_{k-1}}_{{\mathcal{M}_b(I)}}+\alpha_{k}{s(1-s)}\abs{v_{k-1}}_{W^{s,1+s(1-s)}(I)}\right.\\
&\left. +\alpha_{k-1} s(1-s)\abs{\int_I v_{k-1}(x)\,dx}:\right.\\
&\left. v_i\in BV(I)\text{ for }0\leq i\leq k-2,\,\,v_{k-1}\in W^{s,1+s(1-s)}(I).\AAA}
\end{align*}
\end{enumerate}
Moreover, we say that $u$ belongs to the space of \emph{functions with bounded total generalized variation}, and we write $u\in BGV_{\alpha}^{k+s}(I)$ if
\be
\norm{u}_{BGV_{\alpha}^{k+s}(I)}:=\norm{u}_{L^1(I)}+\abs{u}_{TGV_{\alpha}^{k+s}(I)}<+\infty,
\ee
where $0\leq s<1$, $k\in\mathbb N$, $\alpha=(\alpha_0,\alpha_1,\alpha_2,\ldots, \alpha_{k})\in\R_+^{k+1}$. Additionally, we write $u\in BGV^{k+s}(I)$ if there exists $\alpha\in \R_+^{k+1}$ such that $u\in BGV_{\alpha}^{k+s}(I)$. Note that if $u\in BGV_{\alpha}^{k+s}(I)$ for some $\alpha\in \R_+^{k+1}$, then $u\in BGV_{\beta}^{k+s}(I)$ for every $\beta \in \R_+^{k+1}$.
\end{define}
We observe that the $TGV^{k+s}$ seminorm is actually ``intermediate" between the $TGV^k$ seminorm and the $TGV^{k+1}$ seminorm. To be precise, we have the following identification.
\begin{theorem}[Asymptotic behavior of the fractional $TGV$ seminorm-1]\label{intermediate}
For every $u\in BV(I)$, up to the extraction of a (non-relabeled) subsequence there holds 
\be
\lim_{s\nearrow 1}\abs{u}_{TGV^{1+s}_{\alpha}(I)}= \abs{u}_{TGV^2_{\alpha}(I)}\text{ and }\lim_{s\searrow 0}\abs{u}_{TGV^{1+s}_{\alpha}(I)}=\alpha_0\abs{u'}_{\mathcal{M}_b(I)}.
\ee
\end{theorem}

Before proving Theorem \ref{intermediate} we state and prove an intermediate result that will be crucial in determining the asymptotic behavior of the $TGV^{1+s}$ seminorm as $s\nearrow 1$.
\begin{proposition}
\label{prop-int}
Let $u\in W^{1,\infty}(I)$. Then 
\be
\limsup_{s\nearrow 1} (1-s)\abs{u}_{W^{s,1+s(1-s)}(I)}\leq |u'|_{\mathcal{M}_b(I)}.
\ee
\end{proposition}
\begin{proof}
Let $u\in W^{1,\infty}(I)$. Then there exists a constant $L>0$ such that
\be|u(x)-u(y)|\leq L|x-y|^s\ee
for every $x,y\in I$ and every $s\in (0,1)$. Thus
\begin{align*}
&\abs{u}^{1+s(1-s)}_{W^{s,1+s(1-s)}(I)}=\int_I\int_I \frac{|u(x)-u(y)|^{1+s(1-s)}}{|x-y|^{1+s(1+s(1-s))}}\,dx\,dy\\
&\quad\leq L^{s(1-s)}\int_I\int_I \frac{|x-y|^{s^2(1-s)}|u(x)-u(y)|}{|x-y|^{1+s(1+s(1-s))}}\,dx\,dy\\
&\quad=L^{s(1-s)}\int_I\int_I \frac{|u(x)-u(y)|}{|x-y|^{1+s}}\,dx\,dy=L^{s(1-s)}\abs{u}_{W^{s,1}(I)}.
\end{align*}
This implies that
\begin{align*}
&(1-s)\abs{u}_{W^{s,1+s(1-s)}(I)}\leq (1-s)L^{\frac{s(1-s)}{1+s(1-s)}}\abs{u}^{\frac{1}{1+s(1-s)}}_{W^{s,1}(I)}\\
&\quad=L^{\frac{s(1-s)}{1+s(1-s)}}\Big[(1-s)\abs{u}_{W^{s,1}(I)}\Big]^{\frac{1}{1+s(1-s)}}e^{\frac{s(1-s)\log{(1-s)}}{1+s(1-s)}}.
\end{align*}
Therefore, by Theorem \ref{approx_frac_der1} we conclude that
\[
\limsup_{s\nearrow 1}(1-s)\abs{u}_{W^{s,1+s(1-s)}(I)}\leq \limsup_{s\nearrow 1}\Big[(1-s)\abs{u}_{W^{s,1}(I)}\Big]^{\frac{1}{1+s(1-s)}}\leq \abs{u'}_{\mathcal{M}_b(I)}. \qedhere 
\]
\end{proof}

A crucial ingredient in the proof of Theorem \ref{intermediate} is a compactness and lower-semicontinuity result for maps with uniformly weighted averages and $W^{s,1+s(1-s)}$-seminorms.

\begin{proposition}\label{compact_seminorm}
Let $\{s_n\}\subset (0,1)$ be such that $s_n\to \bar{s}$, with $\bar{s}\in (0,1]$. For every $n\in \mathbb{N}$ let $v_n\in W^{s_n,1+s_n(1-s_n)}(I)$ be such that 
\be\label{total_s_bound}
\sup_{n\geq 1}s_n(1-s_n)\left\{\abs{v_n}_{W^{s_n,1+s_n(1-s_n)}(I)}+\Big|\int_I v_n(x)\,dx\Big|\right\}<+\infty.
\ee

%\textcolor{red}{if we change from $s(1-s)$ to $s^2(1-s)$ in Definition \ref{TGV_fractional}, we shall change \eqref{total_s_bound} to $s_n^2(1-s_n)$ as well}

Then, for $\bar{s}\in (0,1)$, there exists $\bar v\in W^{\bar s,1+\bar s(1-\bar s)}(I)$ such that, up to the extraction of a (non-relabeled) subsequence, 
\be
\label{eq:NUM1}
v_n\to \bar v\quad\text{strongly in }L^1(I),
\ee
%and
%\be
%s_n(1-s_n^2(2-s_n))\norm{v_n}_{W^{s_n,1}(I)}\to \bar{s}(1-\bar{s})\norm{v}_{W^{\bar{s},1}(I)}.
%\ee
and
\be\label{par_semi_lsc}
\liminf_{n\to\infty}s_n(1-s_n)\abs{v_n}_{W^{s_n,1+s_n(1-s_n)}(I)} \geq\bar{s}(1-\bar{s})\abs{\bar v}_{W^{\bar{s},1+\bs(1-\bs)}(I)}.
\ee

For $\bar{s}=1$, there exists $\bar v\in BV(I)$ such that, up to the extraction of a (non-relabeled) subsequence, 
\be
\label{eq:NUM3}
v_n-\int_I v_n(x)\,dx\to \bar v\quad\text{strongly in }L^1(I),
\ee

%The analogous statement holds for $\bar{s}=1$, by replacing $W^{\bar{s},1+\bs(1-\bs)}(I)$ with $BV(I)$, and \eqref{par_semi_lsc} with
and
\be
\label{eq:NUM4}
\liminf_{n\to\infty}s_n(1-s_n)\abs{v_n}_{W^{s_n,1+s_n(1-s_n)}(I)} \geq \abs{\bar v'}_{\mathcal{M}_b(I)}.
\ee

\end{proposition}
%\begin{proof}
%\textcolor{red}{This is the crucial difficulty in our argument. Proving this lemma consists in providing a diagonal version of Brezis-Bourgain-Mironescu inequality}
%\end{proof}
\begin{proof}
We first observe that for $x$, $y\in I$, $1\leq p<+\infty$, and $s<t$, we have 
\be\abs{x-y}^{1+sp}> \abs{x-y}^{1+tp}.\ee
 Hence, in view of  \eqref{tv_k_method} there holds
\bes
\abs{u}_{W^{s,p}(I)}\leq\abs{u}_{W^{t,p}(I)}
\ees
for every $u\in W^{t,p}(I)$.\\

Without loss of generality (and up to the extraction of a non-relabeled subsequence) we can assume that the sequences $\{s_n\}$ and $\{s_n(1-s_n)\}$ converge monotonically to $\bs$ and $\bs(1-\bs)$, respectively. According to the value of $\bs$ only 4 situations can arise:

\begin{enumerate}[\underline{Case 1}:]
\item 
$\tfrac12\leq \bs<1$: $s_n\searrow \bar s$ and $s_n(1-s_n)\nearrow \bs(1-\bs)$;
\item[\underline{Case 2}:]
$0< \bs<\tfrac12$: $s_n\searrow \bar s$ and $s_n(1-s_n)\searrow \bs(1-\bs)$;
\item[\underline{Case 3}:]
$\tfrac12<\bs\leq 1$: $s_n\nearrow \bar s$ and $s_n(1-s_n)\searrow \bs(1-\bs)$;
\item[\underline{Case 4}:]
$0<\bs\leq \tfrac12$: $s_n\nearrow \bar s$ and $s_n(1-s_n)\nearrow \bs(1-\bs)$.
\end{enumerate}

\medskip

For convenience of the reader we subdivide the proof into three steps.\\

\noindent\textbf{Step 1}:
We first consider Case 1. By \eqref{total_s_bound} there exists a constant $C$ such that

\be
\label{eq:star-page8}
\sup_{n\geq 1}\left\{\abs{v_n}_{W^{s_n,1+s_n(1-s_n)}(I)}+\Big|\int_I v_n(x)\,dx\Big|\right\}\leq C.
\ee
We point out that the function $f:(0,1)\to \R$, defined as
\be\label{function_f_revised_use} 
f(x):=x-\frac{1}{1+x(1-x)}\quad\text{for every }x\in [0,1],
\ee
is strictly increasing on $[0,1]$. In particular, since $s_n\geq \bar{s}$, there holds $f(s_n)\geq f(\bar{s})$, namely
\be s_n-\frac{1}{1+s_n(1-s_n)}\geq \bar{s}-\frac{1}{1+\bar{s}(1-\bar{s})}.\ee

By applying Theorem \ref{thm:embedding-pq} with $s=s_n$, $r=\bs$, $p=1+s_n(1-s_n)$, and $q=1+\bs(1-\bs)$, we obtain that there exists a constant $C$ such that
\be
\label{need1}
\abs{v_n}_{W^{\bs,1+\bs(1-\bs)}(I)}\leq C \abs{v_n}_{W^{s_n,1+s_n(1-s_n)}(I)}
\ee
for every $n\in \N$. The uniform bound \eqref{eq:star-page8} yields then that there exists a constant $C$ such that
\be
\label{eq:star-star-page8}
\sup_{n\geq 1}\abs{v_n}_{W^{\bs,1+\bs(1-\bs)}(I)}\leq C.
\ee

In view of Theorem \ref{Poincare_TV}, Corollary \ref{sobo_frac_ref}, and estimates \eqref{eq:star-page8} and \eqref{eq:star-star-page8} there exists $\bar v\in W^{\bs,1+\bs(1-\bs)}(I)$ such that, up to the extraction of a (non-relabeled) subsequence, we have
\be
\label{eq:NUM2}
v_n\rightharpoonup \bar v\quad\text{weakly in }W^{\bs,1+\bs(1-\bs)}(I).
\ee
Since $\bs(1+\bs(1-\bs))<1$, and $1<\frac{1+\bs(1-\bs)}{1-\bs(1+\bs(1-\bs))}$, by Theorem \ref{embedding} (1.), the embedding of $W^{\bs,1+\bs(1-\bs)}(I)$ into $L^1(I)$ is compact. Property \eqref{eq:NUM1} follows then by \eqref{eq:NUM2}. \\

By the lower semicontinuity of the $W^{\bs,1+\bs(1-\bs)}(I)$ norm with respect to the weak convergence, and by \eqref{need1} we deduce the inequality
\begin{align*}
\bs(1-\bs)\abs{\bar v}_{W^{\bs,1+\bs(1-\bs)}(I)}&\leq\liminf_{n\to +\infty}\bs(1-\bs)\abs{v_n}_{W^{s_n,1+s_n(1-s_n)}(I)}\\
&= \liminf_{n\to +\infty}s_n(1-s_n)\abs{v_n}_{W^{s_n,1+s_n(1-s_n)}(I)}, 
\end{align*}
which in turn yields \eqref{par_semi_lsc}.\\\\
\textbf{Step 2}: Consider now Case 2. The function $g:(0,1)\to \R$, defined as
\be
g(x):=\frac{1}{1+x(1-x)}\quad\text{for every }x\in (0,1),
\ee
is strictly decreasing in $(0,\tfrac12]$. By the definition of the maps $f$ (defined in \eqref{function_f_revised_use}) and $g$ there holds
\be
g(\bs)+f(\bs)=\bs.
\ee
Therefore, by the monotonicity of $g$ in $(0,\tfrac12]$ there exist $\frac12>\hat{s}>\bar{s}$, and $\lambda\in (0,1)$, such that $g(\hat{s})+f(\bar{s})>\lambda \bar{s}$, namely
\be
\bar{s}-\frac{1}{1+\bar{s}(1-\bar{s})}>\lambda \bar{s}-\frac{1}{1+\hat{s}(1-\hat{s})}.
\ee
By the monotonicity of $f$ on $[0,1]$, and the fact that $s_n\geq \bar{s}$ for every $n\in \N$, we have
\be
s_n-\frac{1}{1+s_n(1-s_n)}=f(s_n)>f(\bs)=\bs-\frac{1}{1+\bs(1-\bs)}>\lambda \bar{s}-\frac{1}{1+\hat{s}(1-\hat{s})}.
\ee
By \eqref{total_s_bound} there exists a constant $C$ such that
\be
\label{eq:starpage9}
\sup_{n\geq 1}\left\{\abs{v_n}_{W^{s_n,1+s_n(1-s_n)}(I)}+\Big|\int_I v_n(x)\,dx\Big|\right\}\leq C.
\ee
Since $\frac12>\hat{s}>\bs$, and $s_n(1-s_n)\searrow \bs(1-\bs)$, there exists $n_0\in \N$ such that
\be
1+s_n(1-s_n)<1+\hat{s}(1-\hat{s})\quad\text{for every }n\geq n_0.
\ee
Hence, choosing $s=s_n$, $r=\lambda\bs$, $p=1+s_n(1-s_n)$, and $q=1+\hat{s}(1-\hat{s})$ in Theorem \ref{thm:embedding-pq}, we deduce that there exists a constant $C$ such that 
\be
\abs{v_n}_{W^{\lambda\bs,1+\hat{s}(1-\hat{s})}(I)}\leq C\abs{v_n}_{W^{s_n,1+s_n(1-s_n)}(I)}
\ee
for every $n\geq n_0$. In particular, \eqref{total_s_bound} yields the uniform bound
\be
\sup_{n\geq 1}\abs{v_n}_{W^{\lambda\bs,1+\hat{s}(1-\hat{s})}(I)}\leq C.
\ee
In view of Theorem \ref{Poincare_TV}, Corollary \ref{sobo_frac_ref}, and estimate \eqref{eq:starpage9} we deduce the existence of a map $\bar v$ such that, up to the extraction of a (non-relabeled) subsequence, 
\be
\label{NUM-MORE}
v_n\rightharpoonup \bar v\quad\text{weakly in }W^{\lambda\bs,1+\hat{s}(1-\hat{s})}(I).
\ee
Since $\lambda\bs(1+\hat{s}(1-\hat{s}))<\hat{s}(1+\hat{s}(1-\hat{s}))<1$, and $1<\frac{1+\hat{s}(1-\hat{s})}{1-\lambda\bs(1+\hat{s}(1-\hat{s}))}$, by Theorem \ref{embedding} (1.) the space $W^{\lambda\bs,1+\hat{s}(1-\hat{s})}(I)$ embeds compactly into $L^1(I)$. Hence, the convergence in \eqref{NUM-MORE} holds also strongly in $L^1(I)$, and \eqref{eq:NUM1} follows. In particular, Fatou's Lemma yields
\be \abs{v}_{W^{\bs,1+\bs(1-\bs)}(I)}^{1+\bs(1-\bs)}\leq \liminf_{n_k\to +\infty}\abs{v_{n_k}}_{W^{s_{n_k},1+s_{n_k}(1-s_{n_k})}(I)}^{1+s_{n_k}(1-s_{n_k})},\ee
which in turn implies \eqref{par_semi_lsc}.\\

\noindent\textbf{Step 3}: We omit the proof of the result in Case 4, and in Case 3 for $\bs<1$, as they follow from analogous arguments. Regarding Case 3 for $\bs=1$, by H\"older inequality we have
\begin{align}
&\int_I\int_I \frac{|v_n(x)-v_n(y)|}{|x-y|^{1+\frac{s_n}{2-s_n}}}\,dx\,dy\\
&\leq\left[\int_I\int_I \left(\frac{|v_n(x)-v_n(y)|}{|x-y|^{1+\frac{s_n}{2-s_n}}}\right)^{1+s_n(1-s_n)}\,dx\,dy\right]^{\frac{1}{1+s_n(1-s_n)}}\\
&\label{eq:star3}=\left[\int_I\int_I \frac{|v_n(x)-v_n(y)|^{1+s_n(1-s_n)}}{|x-y|^{1+\frac{s_n}{2-s_n}+s_n(1-s_n)+\frac{s_n^2}{2-s_n}(1-s_n)}}\,dx\,dy\right]^{\frac{1}{1+s_n(1-s_n)}}.
\end{align}
Now,
\be
1+\frac{s_n}{2-s_n}+s_n(1-s_n)+\frac{s_n^2}{2-s_n}(1-s_n)<1+s_n+s_n^2(1-s_n)
\ee
for $n$ big enough (because $s_n\nearrow 1$). Thus
\be
\frac{1}{|x-y|^{1+\frac{s_n}{2-s_n}+s_n(1-s_n)+\frac{s_n^2}{2-s_n}(1-s_n)}}<\frac{1}{|x-y|^{1+s_n+s_n^2(1-s_n)}}
\ee
for every $x,y\in I$, $x\neq y$, and by \eqref{eq:star3} we obtain
\be
\abs{v_n}_{W^{\frac{s_n}{2-s_n},1}(I)}\leq \abs{v_n}_{W^{s_n,1+s_n(1-s_n)}(I)}^{\frac{1}{1+s_n(1-s_n)}}
\ee
for every $n\in \N$. Property \eqref{total_s_bound} yields the existence of a constant $C$ such that
\be
\label{eq:3dots}
\sup_{n\in \N} (1-s_n)\Big(\abs{v_n}_{W^{\frac{s_n}{2-s_n},1}(I)}+\Big|\int_I v_n(x)\,dx\Big|\Big)\leq C.
\ee
Setting $t_n:=\frac{s_n}{2-s_n}$, there holds $t_n\to 1$ as $n\to +\infty$, and \eqref{eq:3dots} implies
\be
\sup_{n\in \N}(1-t_n)\left\{\abs{v_n}_{W^{t_n,1}(I)}+\Big|\int_I v_n(x)\,dx\Big|\right\}\leq C.
\ee
Properties \eqref{eq:NUM3} and \eqref{eq:NUM4} are then a consequence of \cite[Theorem 4]{bourgain.brezis.mironescu}.
\end{proof}

We now prove Theorem \ref{intermediate}.
\begin{proof}[Proof of Theorem \ref{intermediate}]
Fix $\ep>0$. Let $v_0\in BV(I)$ be such that
\be\abs{u}_{TGV^2_{\alpha}(I)}\geq \alpha_0\abs{u'-v_0}_{\mathcal{M}_b(I)}+\alpha_1|v_0'|_{\mathcal{M}_b(I)}-\ep.\ee
Let $v_0^k\in W^{1,\infty}(I)$ satisfy
\be
\abs{|(v_0^k)'|_{\mathcal{M}_b(I)}-|v_0'|_{\mathcal{M}_b(I)}}<\ep,
\ee
and 
\be
\norm{v_0-v_0^k}_{L^1(I)}<\ep.
\ee
 
In view of Proposition \ref{prop-int} there holds

\begin{align}
\nonumber\limsup_{s\nearrow 1}\abs{u}_{TGV^{1+s}_{\alpha}(I)}&\leq \limsup_{s\nearrow 1}\left\{\alpha_0\abs{u'-sv_0^k}_{\mathcal{M}_b(I)}+\alpha_0 s(1-s)\Big|\int_I v_0^k(x)\,dx\Big|\right.\\
&\nonumber\qquad\left.+\alpha_1 s(1-s)\abs{v_0^k}_{W^{s,1+s(1-s)}(I)}\right\}\\
&\nn\quad\leq \alpha_0\abs{u'-v_0^k}_{\mathcal{M}_b(I)}+\alpha_1\abs{(v_0^k)'}_{\mathcal{M}_b(I)}\\
&\nn\quad\leq \alpha_0\abs{u'-v_0}_{\mathcal{M}_b(I)}+\alpha_1\abs{v_0'}_{\mathcal{M}_b(I)}+(\alpha_0+\alpha_1)\ep\\
&\nn\quad\leq \abs{u}_{TGV^2_{\alpha}(I)}+(\alpha_0+\alpha_1+1)\ep.
\end{align}
The arbitrariness of $\ep$ yields
\be
\label{limsup1}
\limsup_{s\nearrow 1}\abs{u}_{TGV^{1+s}_{\alpha}(I)}\leq \abs{u}_{TGV^2_{\alpha}(I)}.
\ee

To prove the opposite inequality, for every $s\in (0,1)$ let $v_0^s\in W^{s,1+s(1-s)}(I)$ be such that 
\begin{align}
\nn
&\alpha_0\abs{u'-sv_0^s}_{\mathcal{M}_b(I)}+\alpha_1 s(1-s)\abs{v_0^s}_{W^{s,1+s(1-s)}(I)}+\alpha_0 s(1-s)\Big|\int_I v_0^s(x)\,dx\Big|\\
&\label{limsup2}\quad\leq \abs{u}_{TGV^{1+s}_{\alpha}(I)}+s.
\end{align}
In view of \eqref{limsup1} and Proposition \ref{compact_seminorm}, there exists $\tilde{v}\in BV(I)$ such that, up to the extraction of a (non-relabeled) subsequence,
\be
\label{eq:dot}
v_0^s-\int_I v_0^s(x)\,dx\to \tilde{v}\quad\text{strongly in }L^1(I),
\ee
as $s\to 1$, and
\be
\label{eq:2dots}
\lim_{s\nearrow 1}s(1-s)\abs{v_0^s}_{W^{s,1+s(1-s)}(I)}\geq \abs{\tilde{v}'}_{\mathcal{M}_b(I)}.
\ee
Additionally, by \eqref{limsup1} and \eqref{limsup2} there holds
\begin{align*}
s\norm{v_0^s}_{L^1(I)}\leq \abs{u'}_{\mathcal{M}_b(I)}+\norm{u'-sv_0^s}_{\mathcal{M}_b(I)}\leq \abs{u'}_{\mathcal{M}_b(I)}+\abs{u}_{TGV^{1+s}_{\alpha}(I)}+s\leq C
\end{align*}
for every $s\in (0,1)$. Thus, there exists a constant $C$ such that
\be
\lim_{s\to 1}\int_I v_0^s(x)\,dx=C.
\ee
In particular, setting $v:=\tilde{v}+C$, by \eqref{eq:dot} and \eqref{eq:2dots} there holds
\be
v_0^s\to v\quad\text{strongly in }L^1(I),
\ee
and
\be
\lim_{s\to 1}s(1-s)\abs{v_0^s}_{W^{s,1+s(1-s)}(I)}\geq \abs{v'}_{\mathcal{M}_b(I)}.
\ee
Passing to the limit in \eqref{limsup2} we deduce the inequality
\bes
\abs{u}_{TGV^2_{\alpha}(I)}\leq \alpha_0\abs{u'-v}_{\mathcal{M}_b(I)}+\alpha_1\abs{v'}_{\mathcal{M}_b(I)}\leq \liminf_{s\nearrow 1}\abs{u}_{TGV^{1+s}_{\alpha}(I)},
\ees
which in turn implies the thesis.\\

To study the case $s\searrow 0$, we first observe that
\be
\label{preliminary}\sup_{s\in (0,1)}\abs{u}_{TGV^{1+s}_{\alpha}(I)}\leq \alpha_0\AAA |u'|_{\mathcal{M}_b(I)}.
\ee
Thus we only need to prove the opposite inequality. To this aim, for every $s\in (0,1)$ let $v_0^s\in W^{s,1+s(1-s)}(I)$ be such that 
\begin{align}
&\nn\alpha_0\abs{u'-sv_0^s}_{\mathcal{M}_b(I)}+\alpha_1 s(1-s)\abs{v_0^s}_{W^{s,1+s(1-s)}(I)}+\alpha_0 s(1-s)\Big|\int_I v_0^s(x)\,dx\Big|\\
&\quad\label{liminf3}\leq \abs{u}_{TGV^{1+s}_{\alpha}(I)}+s.
\end{align}
Since $s(1+s(1-s))<1$ for $s\in(0,1)$, by \eqref{preliminary} and \eqref{liminf3}, and in view of Theorem \ref{Poincare_TV}, there exists a constant $C$ such that
\be
s\int_I v_0^s(x)\,dx\to C,\quad\text{and}\quad s v_0^s\to C\quad\text{strongly in }L^1(I).
\ee
\AAA
Passing to the limit in \eqref{liminf3} we deduce the inequality
\be
\alpha_0\abs{u'}_{\mathcal{M}_b(I)}\leq \alpha_0|u'-C|_{\mathcal{M}_b(I)}+\alpha_0 C\leq \liminf_{s\searrow 0}\abs{u}_{TGV^{1+s}_{\alpha}(I)}.
\ee
The thesis follows owing to \eqref{preliminary}.
\end{proof}
\begin{corollary}[Asymptotic behavior of the fractional $TGV$ seminorm-2]
\label{cor:higher-order}
Let $k\geq 2$. For every $u\in BV(I)$, up to the extraction of a (non-relabeled) subsequence there holds 
\be
\lim_{s\nearrow 1}\abs{u}_{TGV^{k+s}_{\alpha}(I)}= \abs{u}_{TGV^{k+1}_{\alpha}(I)}\text{ and }\lim_{s\searrow 0}\abs{u}_{TGV^{k+s}_{\alpha}(I)}=\abs{u}_{TGV^{k}_{\hat{\alpha}}(I)},
\ee 
where $\hat{\alpha}:=(\alpha_0,\dots,\alpha_{k-1})\in\R^{k}_{+}$.
\end{corollary}
\begin{proof}
The result follows by straightforward adaptations of the arguments in the proof of Theorem \ref{intermediate}.
\end{proof}
%
%\CCC
%\begin{corollary}\label{small_than_one}
%Let $k=0$. For every $u\in BV(I)$, up to the extraction of a (non-relabeled) subsequence there holds
%\be
%\lim_{s\nearrow 1}\abs{u}_{TGV^s_\alpha} = \abs{u}_{TV_\alpha},
%\ee
%where $\alpha\in\R^+$.
%\end{corollary}
%\AAA
%\begin{proof}
%Again, the result follows by the arguments in the proof of Theorem \ref{intermediate}.
%\end{proof}

We proceed by showing that the minimization problem in Definition \ref{TGV_fractional} has a solution. 

\begin{proposition}\label{equiv_norm}
If the infimum in Definition \ref{TGV_fractional} is finite, then it is attained.
\end{proposition}
\begin{proof}
Let $k=1$. Let $\alpha\in \mathbb{R}^2_{+}$, and let $u\in BGV^{1+s}_{\alpha}(I)$. We need to show that 
\begin{multline}\label{TGV2_miner}
\abs{u}_{TGV^{1+s}_{\alpha}(I)}=\min\flp{\alpha_0\abs{u'-sv}_{{\mathcal{M}_b(I)}}+\alpha_1s(1-s)\abs{v}_{W^{s,1+s(1-s)}(I)}\right.\\
\left. +\alpha_0 s(1-s)\Big|\int_I v(x)\,dx\Big|\AAA:\, v\in W^{s,1+s(1-s)}(I)}.
\end{multline}
We first observe that $u\in BV(I)$. \\\\
Indeed, let $\delta>0$, and let $v\in W^{s,1+s(1-s)}(I)$ be such that 
\be
\alpha_0\abs{u'-sv}_{{\mathcal{M}_b(I)}}+\alpha_1s(1-s)\abs{v}_{W^{s,1+s(1-s)}(I)}+\alpha_0 s(1-s)\Big|\int_I v(x)\,dx\Big|\leq \abs{u}_{TGV^{1+s}_{\alpha}(I)}+\delta.
\ee
By H\"older inequality there holds
\begin{align*}
&\alpha_0 \abs{u'}_{\mathcal{M}_b(I)}\leq \alpha_0 \abs{u'-sv}_{\mathcal{M}_b(I)}+\alpha_0s\norm{v}_{L^1(I)}\\
&\leq \alpha_0 \abs{u'-sv}_{\mathcal{M}_b(I)}+\alpha_1 s\abs{v}_{W^{s,1+s(1-s)}(I)}+\alpha_0s\norm{v}_{L^{1+s(1-s)}(I)}+\alpha_0 s(1-s)\Big|\int_I v(x)\,dx\Big|\\
&\leq \abs{u}_{TGV^{1+s}_{\alpha}(I)}+\delta+\alpha_1 s^2\abs{v}_{W^{s,1+s(1-s)}(I)}+\alpha_0 s\norm{v}_{L^{1+s(1-s)}(I)},
\end{align*}
which implies the claim.\\

Let now $\seqn{v_n}\subset W^{s,1+s(1-s)}(I)$ be a minimizing sequence for \eqref{TGV2_miner}. Since $s(1+s(1-s))<1$ for $s\in (0,1)$, by Theorem \ref{embedding} (1.) there exists a constant $C$ such that
\be
\sup_{n\in \mathbb{N}}\norm{v_n}_{W^{s,1+s(1-s)}(I)}\leq C.
\ee
Thus, by Corollary \ref{sobo_frac_ref} there exists $\bar v\in W^{s,1+s(1-s)}(I)$ such that, up to the extraction of a (non-relabeled) subsequence, there holds
\be
v_n\rightharpoonup \bar v\quad\text{weakly in }W^{s,1+s(1-s)}(I),
\ee
and hence by Theorem \ref{embedding} (1.),
\be
v_n\to \bar v\quad\text{strongly in }L^1(I).
\ee
The thesis follows now by the lower semicontinuity of the total variation and the $W^{s,1+s(1-s)}$-norm with respect to the $L^1$ convergence and the weak convergence in $W^{s,1+s(1-s)}(I)$, respectively. \\\\
For $k=2$, let $\seqn{v_0^n}\subset BV(I)$ and $\seqn{v_1^n}\subset W^{s,1+s(1-s)}(I)$ be such that
\begin{align*}
&\lim_{n\to +\infty}\left\{\alpha_0|u'-v_0^n|_{\mathcal{M}_b(I)}+\alpha_1|(v_0^n)'-sv_1^n|_{\mathcal{M}_b(I)}+\alpha_2 s(1-s)|v_1^n|_{W^{s,1+s(1-s)}(I)}\right.\\
&\quad\left.+\alpha_0 s(1-s)\Big|\int_I v_1^n(x)\,dx\Big|\AAA\right\}=TGV^{2+s}_{\alpha}(I).
\end{align*}
Since $s(1+s(1-s))<1$ for $s\in (0,1)$, by Theorem \ref{embedding} (1.) we obtain that $\seqn{v_1^n}$ is uniformly bounded in $W^{s,1+s(1-s)}(I)$. Therefore, $\seqn{v_0^n}$ is uniformly bounded in $BV(I)$, and there exist $v_0\in BV(I)$ and $v_1\in W^{s,1+s(1-s)}(I)$ such that, up to the extraction of a (non-relabeled) subsequence,
\be
v_0^n\wtos v_0\quad\text{weakly* in }BV(I),
\ee
and
\be
v_1^n\rightharpoonup v_1\quad\text{weakly in }W^{s,1+s(1-s)}(I).
\ee
In particular, by Theorem \ref{embedding} (1.), 
\be
v_1^n\to v_1\quad\text{strongly in }L^1(I).
\ee
The minimality of $v_0$ and $v_1$ is a consequence of lower semicontinuity. The thesis for $k>2$ follows by analogous arguments.
\end{proof}
We observe that the $TGV^{k+s}$ seminorms are all topologically equivalent to the total variation seminorm.

\begin{proposition}\label{equivalent_semi}
For every $k\geq 1$ and $0<s<1$, we have
\be
BV(I)\sim BGV^k(I)\sim BGV^{k+s}(I),
\ee
namely the three function spaces are topologically equivalent.
\end{proposition}
\begin{proof}
We only show that 
\be
\label{equivalent}
BV(I)\sim BGV^{1+s}(I)\sim BGV^2(I).
\ee
The proof of the inequality for $k>1$ is analogous. In view of \eqref{preliminary}, to prove the first equivalence relation in \eqref{equivalent} we only need to show that there exist a constant $C$ and a multi-index $\alpha\in\mathbb{R}^{2}_+$ such that
\be
\abs{u'}_{\mathcal{M}_b\fsp{I}}\leq C\abs{u}_{TGV^{1+s}_{\alpha}(I)}.
\ee
By Theorem \ref{embedding} we have
\begin{align*}
&\abs{u'}_{\mathcal{M}_b(I)}\leq \abs{u'-sv_0}_{\mathcal{M}_b(I)}+s\norm{v_0}_{L^1(I)}\\
&\quad\leq \abs{u'-sv_0}_{\mathcal{M}_b(I)}+Cs\abs{v_0}_{W^{s,1+s(1-s)}(I)}+s(1-s)\Big|\int_I v_0(x)\,dx\Big|\\
&\quad=\abs{u'-sv_0}_{\mathcal{M}_b(I)}+\frac{C}{(1-s)}s(1-s)\abs{v_0}_{W^{s,1+s(1-s)}(I)}+s(1-s)\Big|\int_I v_0(x)\,dx\Big|
\end{align*}
for every $v_0\in W^{s,1+s(1-s)}(I)$. Thus 
\be
\abs{u'}_{\mathcal{M}_b(I)}\leq C\abs{u}_{TGV^{1+s}_{1,\frac{C}{(1-s)}}(I)}
\ee
for every $s\in (0,1)$. This completes the proof of the first equivalence in \eqref{equivalent}. Property \eqref{equivalent} follows now by \cite[Theorem 3.3]{bredies.valkonen.proceeding}.
\end{proof}

\section{The fractional $r$-order $TGV^r$ functional}
\label{sec:Gamma}
In this section we introduce the fractional $r$-order $TGV^r$ functional and prove a $\Gamma$-convergence result with respect to the parameters $\alpha$ and $s$.
\begin{define}
\label{def:fractional TGV functional}
Let $r\in[1,+\infty)$, $\alpha\in \R^{\lfloor r\rfloor+1}_+$, and $u_\eta\in L^2(I)$. We define the functional $\F^{r}_{\alpha}:BV(I)\to [0,+\infty)$ as
\be
\F^{r}_{\alpha}(u):=\int_I |u-u_\eta|^2\,dx+\abs{u}_{TGV^{r}_{\alpha}(I)}
\ee
for every $u\in BV(I)$. Note that the definition is well-posed due to Proposition \ref{equivalent_semi}.
\end{define}
The main result of this section reads as follows.
\begin{theorem}[$\Gamma$-convergence of the fractional order $TGV$ functional]
\label{thm:new-Gamma}
Let $k\in \N$ and $u_\eta\in L^2(I)$. Let $\{s_n\}\subset [0,1]$ and $\{\alpha_n\}\subset \R^{k+1}_+$ be such that $s_n\to s$, and $\alpha_n\to \alpha$. Then, if $s\in (0,1]$, the functional $\F^{k+s_n}_{\alpha_n}$ $\Gamma$-converges to $\F^{k+s}_{\alpha}$ in the weak* topology of $BV(I)$, namely for every $u\in BV(I)$ the following two conditions hold:
\begin{enumerate}
\item[{\rm (LI)}] If 
\be u_n\wtos\text{weakly* in }BV(I),\ee
then 
\be
\F^{k+s}_{\alpha}(u)\leq \liminf_{n\to +\infty}\F^{k+s_n}_{\alpha_n}(u_n).\ee
\item[{\rm (RS)}]
There exists $\{u_n\}\subset BV(I)$ such that 
\be u_n\wtos u\quad\text{weakly* in }BV(I),\ee
and 
\be
\limsup_{n\to +\infty}\F^{k+s_n}_{\alpha_n}(u_n)\leq \F^{k+s}_{\alpha}(u).\ee
\end{enumerate}
The same result holds for $s=0$, by replacing $\alpha=(\alpha_0,\dots,\alpha_k)$ with $\hat{\alpha}:=(\alpha_0,\dots,\alpha_{k-1})$ in {\rm (LI)} and {\rm (RS)}.
\end{theorem}
\begin{remark}
We recall that $\Gamma$-convergence is a variational convergence, originally introduced by E. De Giorgi and T. Franzoni in the seminar paper \cite{degiorgi.franzoni}, which guarantees, roughly speaking, convergence of minimizers of the sequence of functionals to minimizers of the $\Gamma$-limit. The first condition in Theorem \ref{thm:new-Gamma}, known as \emph{liminf inequality} ensures that the $\Gamma$-limit provides a lower bound for the asymptotic behavior of the $TGV^{k+s_n}_{\alpha_n}$ functionals, whereas the second condition, namely the \emph{existence of a recovery sequence} guarantees that this lower bound is attained. We refer to \cite{braides} and \cite{dalmaso} for a thorough discussion of the topic. 
\end{remark}

We subdivide the proof of Theorem \ref{thm:new-Gamma} into two propositions. The next result will be crucial for establishing the liminf inequality.

\begin{proposition}\label{compact_semi_para}
Let $k\in \N$ and $u_\eta\in L^2(I)$. Let $\{s_n\}\subset [0,1]$ and $\{\alpha_n\}\subset \R^{k+1}_+$ be such that $s_n\to s$, and $\alpha_n\to \alpha$.
Let $\{u_n\}\in BV(I)$ be such that
\be\label{BGV_up_bdd}
\sup_{n\in\mathbb N} \F^{k+s_n}_{\alpha_n}(u_n)<+\infty.
\ee
Then, there exists $u\in BV(I)$ such that, up to the extraction of a (non-relabeled) subsequence, there holds
\be\label{BV_conv_w}
u_n\wtos u \text{ in }BV(I).
\ee
In addition, if $s\in (0,1]$ there holds
\be
\label{star}
\abs{u}_{TGV^{k+s}_{{\alpha}}(I)}\leq\liminf_{n\to +\infty}\abs{u_n}_{TGV^{k+s_n}_{\alpha_n}(I)}.
\ee
If $s=0$ we have
\be
\label{star23}
\abs{u}_{TGV^{k}_{\hat{\alpha}}(I)}\leq\liminf_{n\to +\infty}\abs{u_n}_{TGV^{k+s_n}_{\alpha_n}(I)},\ee
where $\hat{\alpha}\in\R^{k}_+$ is the multi-index $\hat{\alpha}:=(\alpha_0,\dots,\alpha_{k-1})$.
\end{proposition}
\AAA
\begin{proof}
We prove the statement for $k=1$. The proof of the result for $k>1$ follows via straightforward modifications. \\

For $k=1$, we have $\seqn{\alpha^n}\subset \R^2$, and 
\be\label{bdd_para}
\fsp{\alpha^n_0,\alpha^n_1}\to \fsp{\alpha_0,\alpha_1}.
\ee
By Proposition \ref{equiv_norm} we deduce that there exists $v_0^n\in W^{s_n,1+s_n(1-s_n)}(I)$ such that
\begin{multline}
\label{eq:un-bd-un}
\abs{u_n}_{TGV_{\alpha^n}^{1+s_n}(I)} = \alpha^n_0\abs{u'_n-s_nv_0^n}_{\mathcal M_b(I)}+\alpha^n_1 s_n(1-s_n)\abs{v_0^n}_{W^{s_n,1+s_n(1-s_n)}(I)}\\
+\alpha_0^n s_n(1-s_n)\Big|\int_I v_0^n(x)\,dx\Big|\AAA.
\end{multline}
We preliminary observe that \eqref{BGV_up_bdd}, \eqref{bdd_para}, and \eqref{eq:un-bd-un} yield the existence of a constant $C$ such that
\be
\label{eq:estimate-A}
 s_n(1-s_n)\Big(\abs{v_0^n}_{W^{s_n,1+s_n(1-s_n)}(I)}+\Big|\int_I v_0^n(x)\,dx\Big|\Big)+\abs{u_n'-s_n v_0^n}_{\mathcal{M}_b(I)}\leq C\abs{u_n}_{TGV_{\alpha^n}^{k+s_n}(I)}\leq C\ee
for every $n\in \N$. For convenience of the reader we subdivide the proof into three steps.\\

\noindent\textbf{Step 1}: Assume first that $s\in (0,1)$. By Proposition \ref{compact_seminorm} there exists $v_0\in W^{s,1+s(1-s)}(I)$ such that
\be
\label{need41}
v_0^n\to v_0\quad\text{strongly in }L^1(I), 
\ee
and
\be
\label{eq:li-case1}
\liminf_{n\to \infty} s_n(1-s_n)\abs{v_0^n}_{W^{s_n,1+s_n(1-s_n)}(I)} \geq s(1-s)\abs{v_0}_{W^{s,1+s(1-s)}(I)}.
\ee
By \eqref{BGV_up_bdd}, \eqref{bdd_para}, \eqref{eq:un-bd-un}, and \eqref{need41} there exists a constant $C$ such that
\be\abs{u'_n}_{\mathcal{M}_b(I)}\leq C,\ee
and hence, again by \eqref{BGV_up_bdd},
\bes
 \sup_{n\in\mathbb N} \norm{u_n}_{BV(I)}\leq C\sup_{n\in\mathbb N} \norm{u_n}_{BGV_{\alpha^n}^{1+s_n}(I)}<+\infty
\ees
which implies \eqref{BV_conv_w}.

By \eqref{BV_conv_w}, \eqref{eq:un-bd-un}, \eqref{need41}, \eqref{eq:li-case1}, and since $0<s<1$, there holds
\begin{align*}
&\liminf_{n\to\infty} \abs{u_n}_{TGV_{\alpha^n}^{1+s_n}(I)}\geq\liminf_{n\to\infty} \alpha^n_0\abs{u'_n-s_nv_0^n}_{\mathcal M_b(I)}\\
&\qquad+\liminf_{n\to\infty}\alpha^n_1 s_n(1-s_n)\abs{v_0^n}_{W^{s_n,1+s_n(1-s_n)}(I)}+\liminf_{n\to\infty}\alpha^n_0 s_n(1-s_n)\abs{\int_I v_0^n(x)\,dx}\\
&\quad\geq \alpha_0\abs{u' -s v_0}_{\mathcal M_b(I)} + \alpha_1s(1-s)\abs{v_0}_{W^{s,1+s(1-s)}(I)}\\
&\label{eq:li-case1-2}\qquad +\alpha_0 s(1-s)\abs{\int_I v_0(x)\,dx}\geq \abs{u}_{TGV_{ \alpha}^{1+s}(I)},
\end{align*}
where in the last inequality we used the definition of the $TGV_{\alpha}^{1+s}$-seminorm. In particular, we deduce \eqref{star}.\\

\noindent\textbf{Step 2}: Consider now the case in which $s=1$. In view of Proposition \ref{compact_seminorm}, estimate \eqref{eq:estimate-A} yields the existence of a map $\tilde{v}_0\in BV(I)$ such that
\be
\label{eq:A}
v_0^n-\int_I v_0^n(x)\,dx\to \tilde{v}_0\quad\text{strongly in }L^1(I)
\ee
and
\be
\label{eq:B}
\abs{\tilde{v}_0'}_{\mathcal{M}_b(I)}\leq \liminf_{n\to +\infty} s_n(1-s_n)\abs{v_0^n}_{W^{s_n,1+s_n(1-s_n)}(I)}.
\ee
On the other hand, by \eqref{eq:estimate-A} there holds
\begin{align*}
\abs{u'_n-s_n\int_I v_0^n(x)\,dx}_{\mathcal{M}_b(I)}\leq \abs{u_n'-s_n v_0^n}_{\mathcal{M}_b(I)}+s_n\norm{v_0^n-\int_I v_0^n(x)\,dx}_{L^1(I)}\leq C
\end{align*}
for every $n\in N$. Thus, there exists $\tilde{u}\in BV(I)$ such that
\be
\label{eq:new-star}
u_n-u_n(0)-s_n\Big(\int_I v_0^n(x)\,dx\Big)x\wtos \tilde{u}\quad\text{weakly* in }BV(I). 
\ee

By \eqref{BGV_up_bdd} and by Definition \ref{TGV_fractional}, we have 
\be
\label{eq:starpage16}
\sup_{n\geq 1}\norm{u_n}_{L^1(I)}\leq C.
\ee
In view of \eqref{eq:new-star}, testing the map $u_n-u_n(0)-s_n\Big(\int_I v_0^n(x)\,dx\Big)x$ against the function $x-\frac12$, we obtain that
\be
\lim_{n\to +\infty}\int_I u_n(x)\left(x-\frac12\right)\,dx+\frac{s_n}{6}\int_I v_0^n(x)\,dx=\int_I \tilde{u}(x)\left(x-\frac12\right)\,dx.
\ee
Thus, by \eqref{eq:starpage16},
\be
\label{eq:C}
\frac{s_n}{6}\Big|\int_I v_0^n(x)\,dx\Big|\leq \abs{\int_I u_n(x)\left(x-\frac12\right)\,dx+\frac{s_n}{6}\int_I v_0^n(x)\,dx}+\frac12\norm{u_n}_{L^1(I)}\leq C
\ee

for every $n\in \N$. By combining \eqref{eq:A}, \eqref{eq:B}, \eqref{eq:new-star}, and \eqref{eq:C} we deduce that there exists $v_0\in BV(I)$ such that
\be
\label{eq:D}
v_0^n\to v_0\quad\text{strongly in }L^1(I),
\ee
and
\be
\label{eq:E}
\abs{v_0'}_{\mathcal{M}_b(I)}\leq \liminf_{n\to +\infty} s_n(1-s_n)\abs{v_0^n}_{W^{s_n,1+s_n(1-s_n)}(I)}.
\ee
In particular, by combining \eqref{eq:estimate-A}, \eqref{eq:D}, and \eqref{eq:E} we have that
\be\abs{u_n'}_{\mathcal{M}_b(I)}\leq C\ee
for every $n\in \N$, which by \eqref{BGV_up_bdd}
 yields \eqref{BV_conv_w}.\\

\noindent\textbf{Step 3}: Consider finally the case in which $s=0$. In view of \eqref{eq:estimate-A}, and by Theorem \ref{Poincare_TV} there holds 
\be
\label{eq:new-dot}
\norm{v_0^n-\int_I v_0^n(x)\,dx}_{L^1(I)}\leq C
\ee
for every $n\in \N$. On the other hand, \eqref{eq:estimate-A} yields
\be
\label{eq:new-2dots}
s_n\Big|\int_I v_0^n(x)\,dx\Big|\leq C
\ee
for every $n\in \N$. Combining \eqref{eq:new-dot} and \eqref{eq:new-2dots} we conclude that there exists a constant $\lambda\in \R$ such that, 
up to the extraction of a (non-relabeled) subsequence there holds
\be
\label{eq:new-3dots}
s_n v_0^n\to \lambda\quad\text{strongly in }L^1(I).
\ee
As a result, in view of \eqref{BGV_up_bdd}, Definition \ref{TGV_fractional}, and \eqref{eq:estimate-A} we deduce the existence of a map $u\in BV(I)$ such that, up to the extraction of a (non-relabeled) subsequence,
\be
\label{eq:NUMP13}
u_n\wtos u\quad\text{weakly* in }BV(I).
\ee
Hence, by \eqref{eq:un-bd-un}, \eqref{eq:new-3dots}, and \eqref{eq:NUMP13} we have
\begin{align}
&\liminf_{n\to\infty} \abs{u_n}_{TGV_{\alpha^n}^{1+s_n}(I)}\geq\liminf_{n\to\infty}\alpha^n_0\abs{u'_n-s_nv_0^n}_{\mathcal M_b(I)}\\
&+\liminf_{n\to\infty}\alpha^n_1 s_n(1-s_n)\abs{v_0^n}_{W^{s_n,1+s_n(1-s_n)}(I)}+\liminf_{n\to\infty}\alpha^n_0 s_n(1-s_n)\abs{\int_I v_0^n(x)\,dx}\\
&\geq \alpha_0\abs{u'-\lambda}_{\mathcal M_b(I)}+\alpha_0|\lambda|\geq \alpha_0\abs{u'}_{\mathcal M_b(I)},
\end{align}
which in turn implies \eqref{star23}. This concludes the proof of the proposition.
\end{proof}
The following result is instrumental for the construction of a recovery sequence.

\begin{proposition}\label{new_equal}
Let $k\in \N$ and $u_\eta\in L^2(I)$. Let $\{s_n\}\subset [0,1]$ and $\{\alpha_n\}\subset \R^{k+1}_+$ be such that $s_n\to s$, and $\alpha_n\to \alpha$. Let $u\in BV(I)$. Then, if $s\in (0,1]$ there holds
\be\label{new_equal1}
\limn \abs{u}_{TGV^{k+s_n}_{\alpha_n}(I)} = \abs{u}_{TGV^{k+s}_{\alpha}(I)},
\ee
and if $s=0$,
\be\label{new_equal2}
\limn \abs{u}_{TGV^{k+s_n}_{\alpha_n}(I)} = \abs{u}_{TGV^{k+s}_{\hat{\alpha}}(I)},
\ee
where $\hat{\alpha}\in\R^{k}_+$ is the multi-index $\hat{\alpha}:=(\alpha_0,\dots,\alpha_{k-1})$.
\end{proposition}
\begin{proof}
We prove the proposition for $k=1$. The thesis for $k>1$ can be proven by analogous arguments. The special cases $s=0$ and $s=1$ have already been analyzed in Theorem \ref{intermediate}, in the situation in which $\alpha_n=\alpha$ for all $n\in \N$. The thesis for general sequences $\{\alpha_n\}$ follows by straightforward adaptations. Therefore, without loss of generality, we can assume that $s\in (0,1)$.
As a consequence of Proposition \ref{compact_semi_para} we immediately have that
\be\label{new_equal_liminf}
\liminf_{n\to\infty}\abs{u}_{TGV^{1+s_n}_{\alpha}(I)} \geq \abs{u}_{TGV^{1+s}_{\alpha}(I)}.
\ee
To prove the opposite inequality, fix $\e>0$. By the definition of the $TGV^{1+s}_{\alpha}(I)$-seminorm, and in view of \cite[Theorems 2.4 and 5.4]{dinezza.palatucci.valdinoci}, there exists $v_0\in C^{\infty}(I)\cap W^{s,1+s(1-s)}(I)$ such that 
\be
\abs{u}_{TGV^{1+s}(I)}\geq \alpha_0{\abs{u'- s v_0}_{\mathcal{M}_b(I)}+\alpha_1{s(1-s)}\abs{v_0}_{W^{s,1+s(1-s)}(I)}}+\alpha_0{s(1-s)}\abs{\int_{I}v_0(x)\,dx}-\e.
\ee

In particular, there holds $v_0\in C^{\infty}(I)\cap W^{s_n,1+s_n(1-s_n)}(I)$ for every $n\in \N$. Hence 
\begin{align}
&\abs{u}_{TGV^{1+s_n}_{\alpha}(I)}\leq  \alpha_0\abs{u'- s_nv_0}_{\mathcal{M}_b(I)}\\
&\quad+\alpha_1{s_n(1-s_n)}\abs{v_0}_{W^{s_n,1+s_n(1-s_n)}(I)}+\alpha_0{s_n(1-s_n)}\abs{\int_{I}v_0(x)\,dx},
\end{align}
for every $n\in \N$, and
\be\label{limsup-new}
\limsup_{n\to\infty}\abs{u}_{TGV^{1+s_n}_{\alpha}(I)}\leq \abs{u}_{TGV^{1+s}_{\alpha}(I)}+\e.
\ee
The thesis follows by the arbitrariness of $\e$, and by combining \eqref{new_equal_liminf} and \eqref{limsup-new}.
\end{proof}
We conclude this section by proving Theorem \ref{thm:new-Gamma}.
\begin{proof}[Proof of Theorem \ref{thm:new-Gamma}]
Property {\rm (LI)} is a direct consequence of Proposition \ref{compact_semi_para}. Property {\rm (RS)} follows by Proposition \ref{new_equal}, choosing $u_n=u$ for every $n\in \N$.
\end{proof}

%
%\begin{lemma}
%Let $u_n\in BGV_{\alpha^n}^{k+s}$ be given where $k\in \mathbb N$, $s\in(0,1)$, $\seqn{\alpha^n}\subset \R^+_{k+1}$ is such that $\limn\norm{\alpha^n}_{l^\infty}=\infty$, and
%\be\label{BGV_up_bdd}
%\sup_{n\in\mathbb N} \flp{\norm{u_n}_{BGV_{\alpha^n}^{k+s}(I)}}<+\infty.
%\ee
%Then, up to a subsequence, we have 
%\be\label{BV_conv_w}
%u_n\wtos u \text{ in }BV
%\ee
%such that $u\in BGV_{\tilde \alpha'}^{k'+s'}$ where $k'\leq k$, $s'\in\flp{0,s}$, and $\tilde \alpha':=(\tilde \alpha'_1,\ldots,\tilde \alpha'_{k+1})$ if $\tilde \alpha_i<+\infty$ where $\tilde \alpha_i:=\limn \alpha_i^n$.
%\end{lemma}
%\begin{proof}
%We only deal with the case that $k=1$ and we remark that the other cases can be dealt similarly.
%\begin{enumerate}[1.]
%\item
%assume that $\tilde \alpha_1=\limn\alpha^n_0\to\infty$ but $\tilde \alpha_2,\tilde \alpha_3<+\infty$. Then by \eqref{fractional_TGV} and \eqref{BGV_up_bdd} we may force $v_0=u'$ and hence we have 
%\be
%\alpha_1\abs{u_n'-v_1}_{\mathcal M_b(I)} + \alpha_2\abs{v_1}_{W^{1,s}(I)}
%\ee
%
%\end{enumerate}
%\end{proof}
%
%
\section{The bilevel training scheme equipped with $TGV^r$ regularizer}\label{BLSFRAC}
Let $r\in[1,+\infty)$ be given and recall ${\lfloor r\rfloor}$ denote the largest integer smaller than or equal to $r$. We propose the following training scheme $(\mathcal R)$ which takes into account the order of derivation $r$ of the regularizer and the parameter $ \alpha\in \R^{\lfloor r\rfloor+1}_+$ simultaneously. We restrict our analysis to the case in which $\alpha$ and $r$ satisfy the \emph{box constraint} \be\label{box_const}
(\alpha, r)\in [P,1/P]^{\lfloor r\rfloor +1}\times[1,1/P]
\ee
where $P\in(0,1)$ small is a fixed real number.\\

Our new training scheme $(\mathcal R)$ is defined as follows:
\begin{flalign}
\text{Level 1. }&\,\,\,\,\,\,\,\,\,\,\,\,\,\,\,\,\,\,\,\,\,\,\,\,\,\,\,\,(\tilde{\alpha},\tilde{r}):=\argmin\flp{\mathcal I(\alpha,r),\,\,(\alpha, r)\in [P,1/P]^{\lfloor r\rfloor +1}\times[1,1/P]},\label{frac_para_bi}&\\
\text{Level 2. }&\,\,\,\,\,\,\,\,\,\,\,\,\,\,\,\,\,\,\,\,\,\,\,\,\,\,\,\,u_{\ta,r}:=\argmin\flp{\F^r_{\alpha}(u):\,\,{u\in BV(I)}}\label{frac_para_bi2}&
\end{flalign}
%\begin{enumerate}[Level 1.]
%\item
%\be\label{frac_para_bi}
%(\bar{\alpha},\bar{r}):=\argmin\flp{\norm{u_{\alpha,r}-u_c}_{L^2(I)}^2,\,\,(\alpha, r)\in [P,1/P]^{\lfloor r\rfloor +2}},
%\ee
%\item
%\be\label{frac_para_bi2}
%u_{\ta,r}:=\argmin_{u\in BGV_{\alpha}^{r}(I)}\flp{\norm{u-u_\eta}_{L^2(I)}^2+\abs{u}_{TGV_{\alpha}^{r}(I)}},
%\ee
%\end{enumerate}
where $\mathcal I(\alpha,r)$, defined as
\be\label{define_cost_fun}
\mathcal I(\alpha,r):=\norm{u_{\alpha,r}-u_c}_{L^2(I)}^2
\ee
is the \emph{cost function} associated to the training scheme $(\mathcal R)$, $\F^r_{\alpha}$ is introduced in Definition \ref{def:fractional TGV functional}, the map $u_c\in L^2(I)$ represents a noise-free test signal, and $u_\eta\in L^2(I)$ is the noise (corrupted) signal.\\

Note that we only allow the parameters $\ta$ and the order $r$ of regularizers to lie within a prescribed finite range. This is needed to force the optimal reconstructed signal $u_{\tilde{\alpha},\tilde{r}}$ to remain inside our proposed space $BV(I)$ (see Proposition \ref{compact_semi_para}). In particular, if some of the components of $\tilde{\alpha}$ blow up to $\infty$, we might end up in the space $W^{r,1}(I)$, which is outside the purview of this paper.
%The extension of our analysis to such spaces will be the subject of our follow-up work \cite{liu2016Image}. 
We point out that $P$ can be chosen as small as the user wants. Thus, despite the box constraint, our analysis still incorporates a large class of regularizers, such as $TV$ and $TGV^2$ (see, e.g., \cite{2015arXiv150807243D}).\\
%Finally, in Corollary \ref{equal_para1}, we show that if the parameter $\ta$ satisfies $\abs{\ta_n}_{l^\infty}\to 0$ if and of one of its components $\alpha_i^n\to 0$, then we can relax $\hat \alpha=0$ in the box constraint \eqref{box_const}.\\\\
%

Before we state the main theorem of this section, we prove a technical lemma and show that \eqref{frac_para_bi2} has a unique minimizer for each given $(\alpha,r)$. 
\begin{lemma}
\label{ex-sol-sec-lev}
For every $r\in [1,\mathrm{1/P}]$ and $\alpha\in \R^{\lfloor r\rfloor+1}_+$, there exists a unique $u_{\alpha,r}\in BV(I)$ solving the minimization problem \eqref{frac_para_bi2}. 
\end{lemma}
\begin{proof}
Let $\seqn{u_n}\subset BV(I)$ be a minimizing sequence for \eqref{frac_para_bi2}. By Proposition \ref{equivalent_semi}, $\seqn{u_n}$ is uniformly bounded in $BV(I)$. Thus there exists ${u}_{\alpha,r}\in BV(I)$ such that
\be u_n\wtos {u}_{\alpha,r}\quad\text{weakly* in }BV(I),\ee
and hence also strongly in $L^2(I)$. The minimality of ${u}_{\alpha,r}$ follows then by Proposition \ref{compact_semi_para}, whereas the uniqueness of the minimum is a consequence of the strict convexity of the functional.
\end{proof}

\begin{comment}
\RRR
\begin{lemma}
Let $u_\eta$ and $u_c\in BV(I)$ be given such that 
\be\label{noise_bigger_lemma}
TV(u_\eta)>TV(u_c).
\ee
Then there exists $\hat\alpha'>0$ such that 
\be\label{alittle_better}
\int_I \abs{u_{\hat\alpha'}-u_c}^2dx<\int_I \abs{u_{0}-u_c}^2dx,
\ee
where 
\be
u_\alpha:=\argmin_{u\in BV(I)}\flp{\norm{u-u_\eta}_{L^2(I)}^2+\alpha TV(u)}.
\ee
\end{lemma}
\begin{proof}
We observe that, for any $\alpha>0$,
\begin{multline}\label{way_compute_drop_0}
\norm{u_\eta - u_c}_{L^2(I)}^2-\norm{u_\alpha - u_c}_{L^2(I)}^2 \\
= 2\int_I\fmp{u_\eta-u_\alpha}\fmp{u_\alpha-u_c}dx+\norm{u_\eta-u_\alpha}_{L^2(I)}^2\\
= -2\alpha\int_I\fmp{u_\alpha-u_c}d\fsp{\frac{u'_\alpha}{\abs{u'_\alpha}}}+\norm{u_\eta-u_\alpha}_{L^2(I)}^2\\
= 2\alpha\int_I\fsp{\frac{u'_\alpha}{\abs{u'_\alpha}}}D\fmp{u_\alpha-u_c}dx+\norm{u_\eta-u_\alpha}_{L^2(I)}^2 \\
=2\alpha\fsp{TV(u_\alpha) - \int_I\fsp{\frac{u'_\alpha}{\abs{u'_\alpha}}}D(u_c)\,dx}+\norm{u_\eta-u_\alpha}_{L^2(I)}^2\\
\geq 2\alpha\fsp{TV(u_\alpha) -TV(u_c)}+\norm{u_\eta-u_\alpha}_{L^2(I)}^2.
\end{multline}
By \eqref{noise_bigger_lemma} we may choose $\hat\alpha'>0$ small enough such that 
\be
TV(u_{\hat\alpha'}) -TV(u_c)>0.
\ee
Hence, in the view of \eqref{way_compute_drop_0}, we have
\be
\norm{u_\eta - u_c}_{L^2(I)}^2-\norm{u_{\hat\alpha'}- u_c}_{L^2(I)}^2 \geq 2\alpha\fsp{TV(u_{\hat\alpha'}) -TV(u_c)}+\norm{u_\eta-u_{\hat\alpha'}}_{L^2(I)}^2>0,
\ee
which implies \eqref{alittle_better} as desired.
\end{proof}
\end{comment}

\AAA
The next result guarantees existence of solutions to our training scheme.

\begin{theorem}\label{main_result}
%Let $u_\eta$ and $u_c\in BV(I)$ be given such that 
%\be\label{noise_bigger}
%TV(u_\eta)>TV(u_c).
%\ee
Let $u_\eta, u_c\in BV(I)$ be given. Under the box constraint \eqref{box_const}, the training scheme $(\mathcal R)$ admits at least one solution $(\tilde{\alpha},\tilde{r})\in [P,1/P]^{\lfloor \tilde{r}\rfloor +1}\times[1,1/P]$ and provides an associated optimally reconstructed signal $u_{\tilde{\alpha},\tilde{r}}\in BV(I)$.
% provided that $\hat \alpha\leq \hat\alpha'$ where $\hat\alpha'$ is determined in \eqref{alittle_better}.
\end{theorem}
\begin{proof}
The result follows by Theorem \ref{thm:new-Gamma} and by standard properties of $\Gamma$-convergence. We highlight the main steps for convenience of the reader. \\\\
Let $\seqn{(\ta_n,r_n)}$ be a minimizing sequence for \eqref{frac_para_bi} with $\seqn{r_n}\subset [1,1/P]$. Then, up to the extraction of a (non-relabeled) subsequence, there exists $\tilde r\in[1,1/P]$ such that $r_n\to\tilde r$, and $\tilde N\in\N$ large enough such that for all $n\geq \tilde N$, one the following two cases arises.
\begin{enumerate}[1.]
\item
 $r_n\in[\lfloor \tilde r \rfloor,\lfloor \tilde r \rfloor+1]$, or
 \item
  $r_n\in[\lfloor \tilde r \rfloor -1,\lfloor \tilde r \rfloor]$.
 \end{enumerate}
Suppose statement 1 holds (statement 2 can be handled by an analogous argument), then for $n\geq \tilde N$, we have $(\ta_n,r_n)\subset [P,1/P]^{\lfloor \tilde r\rfloor +1}\times[\lfloor \tilde r \rfloor,\lfloor \tilde r \rfloor+1]$ for all $n\geq \tilde N$, and, up to extracting a further subsequence, $\alpha_n\to\tilde \alpha$ with $\tilde \alpha\in [P,1/P]^{\lfloor \tilde r\rfloor +1}$. Let $u_{\alpha_n,r_n}$ be the unique solution to \eqref{frac_para_bi2} provided by Lemma \ref{ex-sol-sec-lev}. By \eqref{frac_para_bi2} and Proposition \ref{equivalent_semi}, there holds
\bes
\norm{u_{\alpha_n,r_n}}_{BGV^{r_n}_{\alpha_n}(I)}\leq \abs{u_\eta}_{TGV^{r_n}_{\alpha_n}(I)}\leq C\abs{u_\eta'}_{\mathcal{M}_b(I)}
\ees
for all $n\geq \tilde N$. 
%There exists $m\in \N$ with $m\in [1,\mathrm{1/P}]$, and $\tilde{r}\in [1,\mathrm{1/P}]$ such that, up to the extraction of a (non-relabeled) subsequence, there holds
%$\lfloor r_n\rfloor\to m$, and $r_n\to \tilde r$. 
%
%We distinguish two cases.
%\begin{enumerate}[C{a}se 1.]
%\item
%$\tilde{r}\notin \N$. Then $m=\lfloor \tilde r\rfloor$. For $n$ big enough we have $\alpha_n\in [\mathrm{P},\mathrm{1/P}]^{m+1}$, and $r_n=m+s_n$, with $s_n\to \tilde{s}$, and $\tilde s\in (0,1)$.
%\item
%$\tilde r \in \N$. Then  either $m=\tilde r$ or $m=\tilde r-1$. In the first scenario, up to the extraction of a (non-relabeled) subsequence, there holds $\lfloor r_n\rfloor=\tilde r$, and $r_n=\tilde r+s_n$, with $s_n\to 0$, whereas in the latter one, $r_n=\tilde r-1+s_n$, with $s_n\to 1$. In both scenarios, for $n$ big enough there holds $\alpha_n\in [\mathrm{P},\mathrm{1/P}]^{m+1}$.
%\end{enumerate}
%
Next, in view of Proposition \ref{compact_semi_para}, there exists a $\tilde{u}\in BGV^{\tilde{r}}_{\tilde{\alpha}}(I)$ 
such that
\bes
u_{\alpha_n,r_n}\wtos \tilde{u}\quad\text{weakly* in }BV(I).
\ees
Thus, in particular, strongly in $L^2(I)$.\\

We claim that 
\bes
\tilde{u}=u_{\tilde{\alpha},\tilde{r}}=\argmin\flp{\F^{\tilde{r}}_{\tilde{\alpha}}(u):\,\,u\in BV(I)}.
\ees
Indeed, by Propositions \ref{compact_semi_para} and \ref{new_equal} there holds
\begin{multline}
 \int_I |\tilde{u}-u_\eta|^2\,dx+\abs{\tilde{u}}_{TGV_{\tilde{\alpha}}^{\tilde{r}}(I)}\leq \liminf_{n\to +\infty}\int_I |u_{\alpha_n, r_n}-u_\eta|^2\,dx+\abs{u_{\alpha_n, r_n}}_{TGV_{\alpha_n}^{r_n}(I)}\\
 \leq \lim_{n\to +\infty}\fmp{\int_I |v-u_\eta|^2\,dx+\abs{v}_{TGV_{\alpha_n}^{r_n}(I)}}\leq \int_I\abs{v-u_\eta}^2dx+\limsup_{n\to\infty}\abs{v}_{TGV_{\alpha_n}^{r_n}(I)}\\
 =\int_I |v-u_\eta|^2\,dx+\abs{v}_{TGV_{\tilde{\alpha}}^{\tilde{r}}(I)}
\end{multline}
for every $v\in BV(I)$, where at the last equality we used \eqref{new_equal1} or \eqref{new_equal2}. This completes the proof of the claim and of the theorem.
\end{proof}

\begin{remark}\label{pract_box_c}
The box constraint \eqref{box_const} is only used to guarantee that a minimizing sequence $\seqn{(\alpha_n,r_n)}$ has a convergent subsequence whose limit is bounded away from $0$. Alternatively, different box-constraints for each parameter $\alpha$ and $r$ might be enforced, such as 
\be
(\alpha,r)\in[P_1,Q_1]\times [P_2,Q_2]\times\cdots\times [1,Q_{\lfloor r\rfloor+2}]
\ee
where $0<P_i<Q_i<+\infty$, $i=1,\ldots, {\lfloor r\rfloor+1}$ and $1<Q_{\lfloor r\rfloor+2}<+\infty$.
\end{remark}

\section{Examples and insight}\label{num_sim}
In order to gain further insight into the cost function $\mathcal I(\alpha,r)$, defined in \eqref{define_cost_fun}, we compute it for a grid of values of $\alpha$ and $r$. We perform this analysis for two signals presenting different features, namely for a signal exhibiting corners (see Figure \ref{fig:nois_fr}) and for a signal with flat areas (see Figure \ref{fig:flat_n_s}). In both cases, for simplicity, we assume $\alpha_0=\alpha_1=\alpha$ and we consider the discrete box-constraint 
\be\label{parameter_domain}
(\alpha, r)\in \flp{0,\,0.005,\,0.01,\,0.015,\ldots,2.5}\times\flp{1,\,1.0025,\,1.005,\,\ldots,\,2}
\ee
(see Remark \ref{pract_box_c}). The reconstructed signal $u_{\alpha,r}$ in \eqref{frac_para_bi2} is computed by using the primal-dual algorithm presented in \cite{chambolle2011first} and \cite{garnett2007image}.\\

\begin{figure}[!h]
\begin{subfigure}{.5\textwidth}
  \centering
  \includegraphics[width=\linewidth]{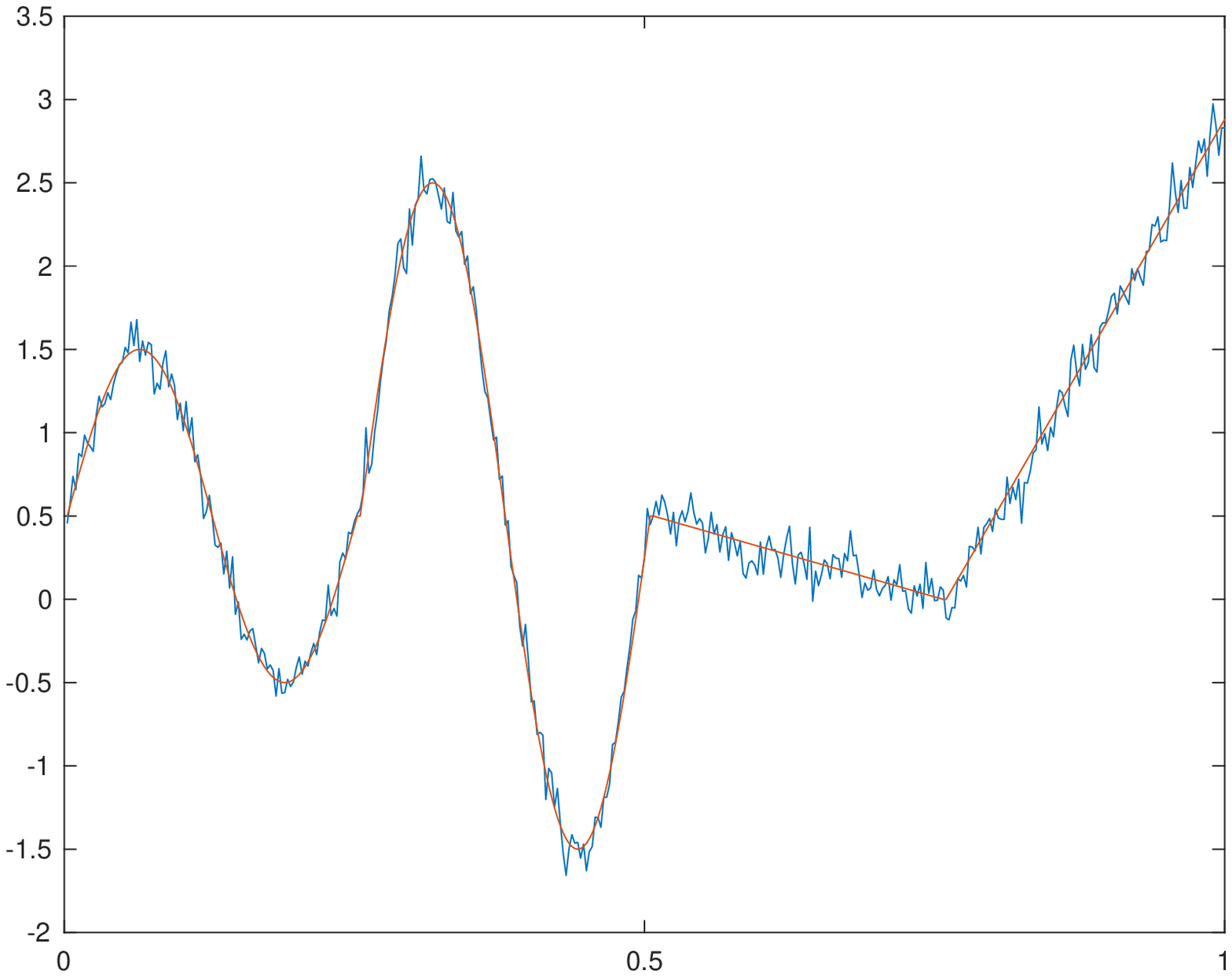}
  \caption{$u_c$ in red and $u_\eta=u_c+\eta$ in blue.}
  \label{fig:nois_fr}
\end{subfigure}%
\begin{subfigure}{.5\textwidth}
  \centering
  \includegraphics[width=\linewidth]{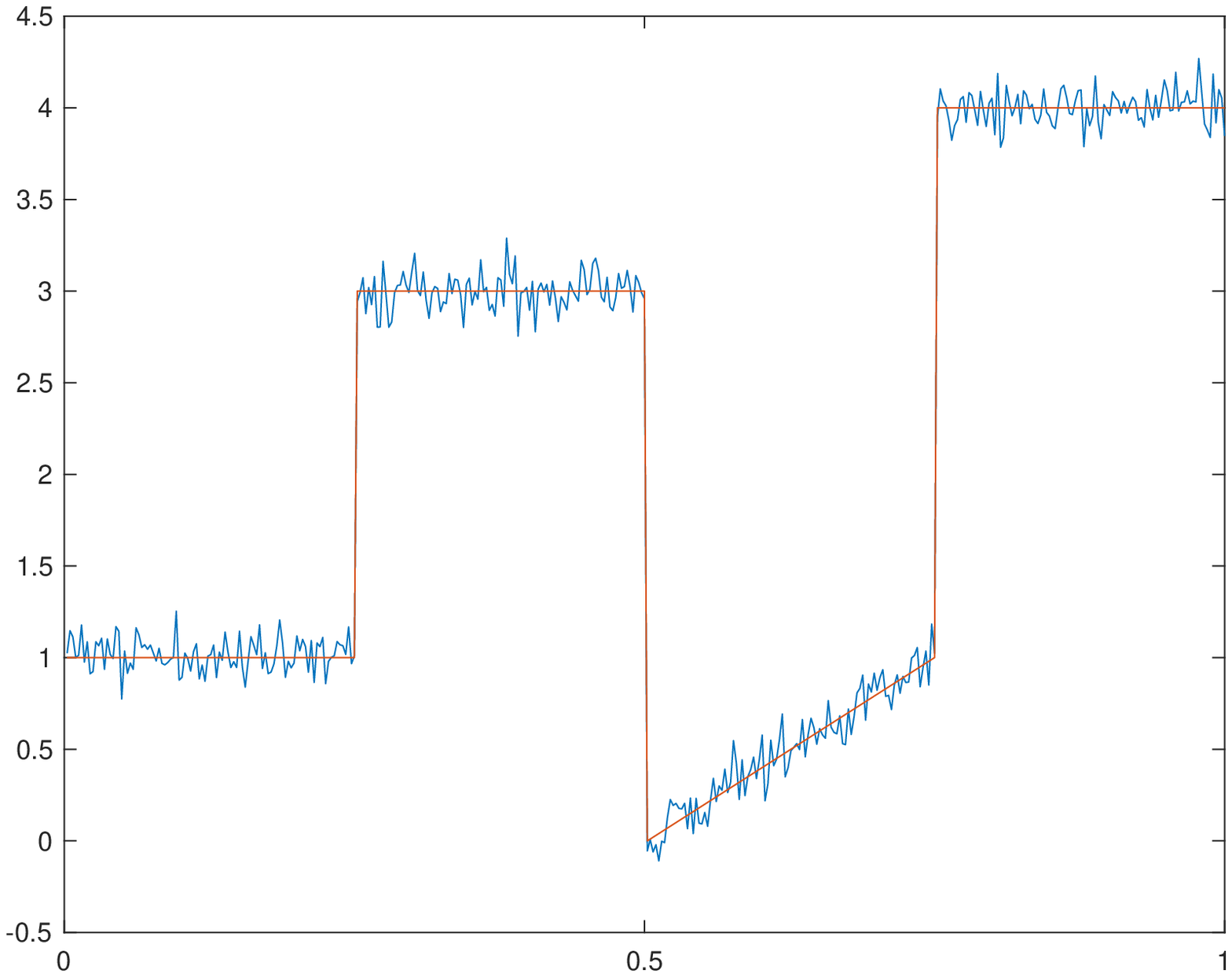}
  \caption{$u_c$ in red and $u_\eta=u_c+\eta$ in blue.}
  \label{fig:flat_n_s}
\end{subfigure}
\caption{The artificial noise $\eta$ is generated by using a Gaussian noise distribution.}
\label{fig:fig1}
\end{figure}

The numerical landscapes of the cost function $\mathcal I(\alpha,r)$ are visualized in Figure \ref{fig:contour} and Figure \ref{fig:flat_cont}, respectively. 

\begin{figure}[!h]
\begin{subfigure}{.5\textwidth}
  \centering
  \includegraphics[width=\linewidth]{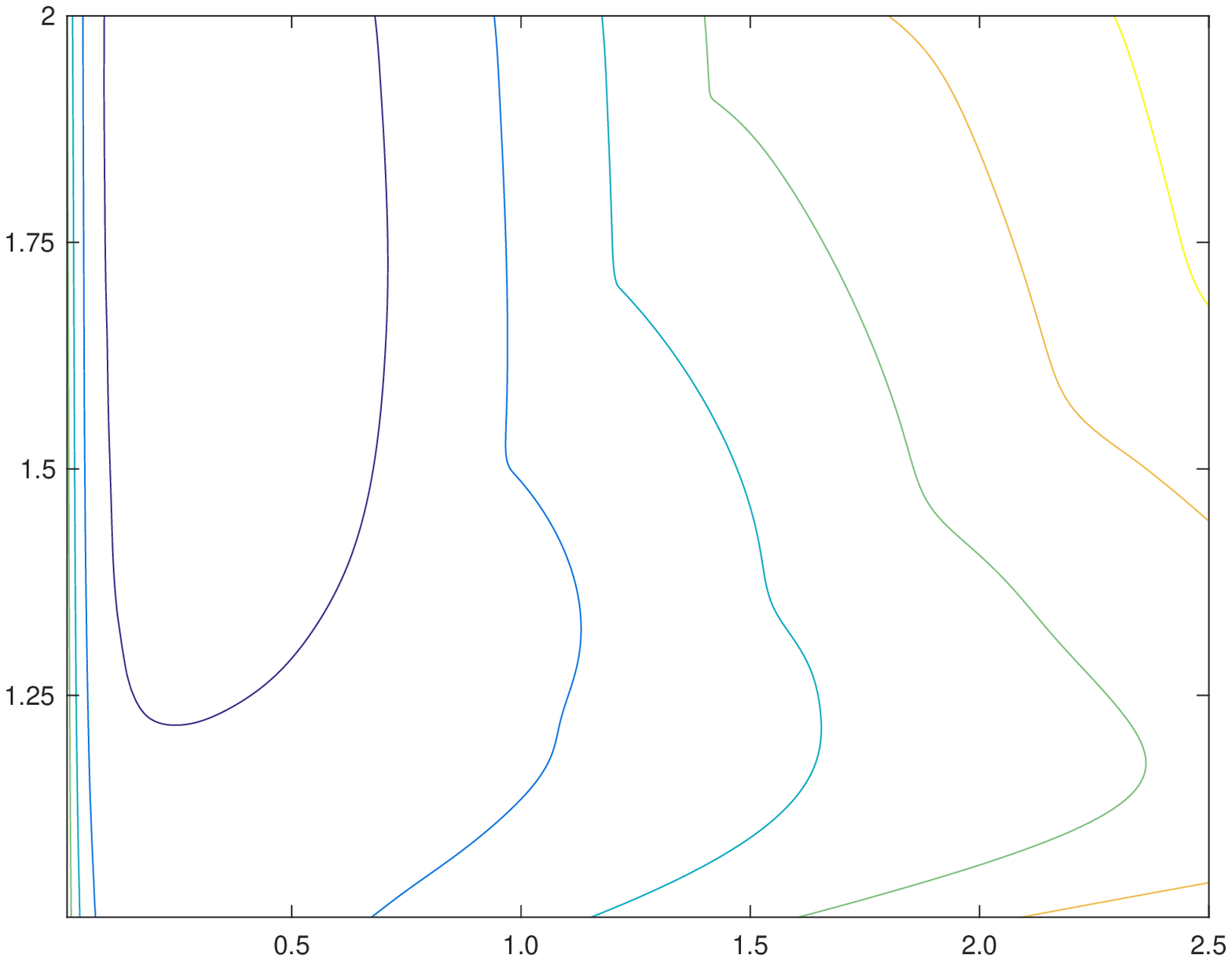}
  \caption{optimal $(\tilde\alpha,\tilde r)=(0.29,1.97)$}
  \label{fig:contour}
\end{subfigure}%
\begin{subfigure}{.5\textwidth}
  \centering
  \includegraphics[width=\linewidth]{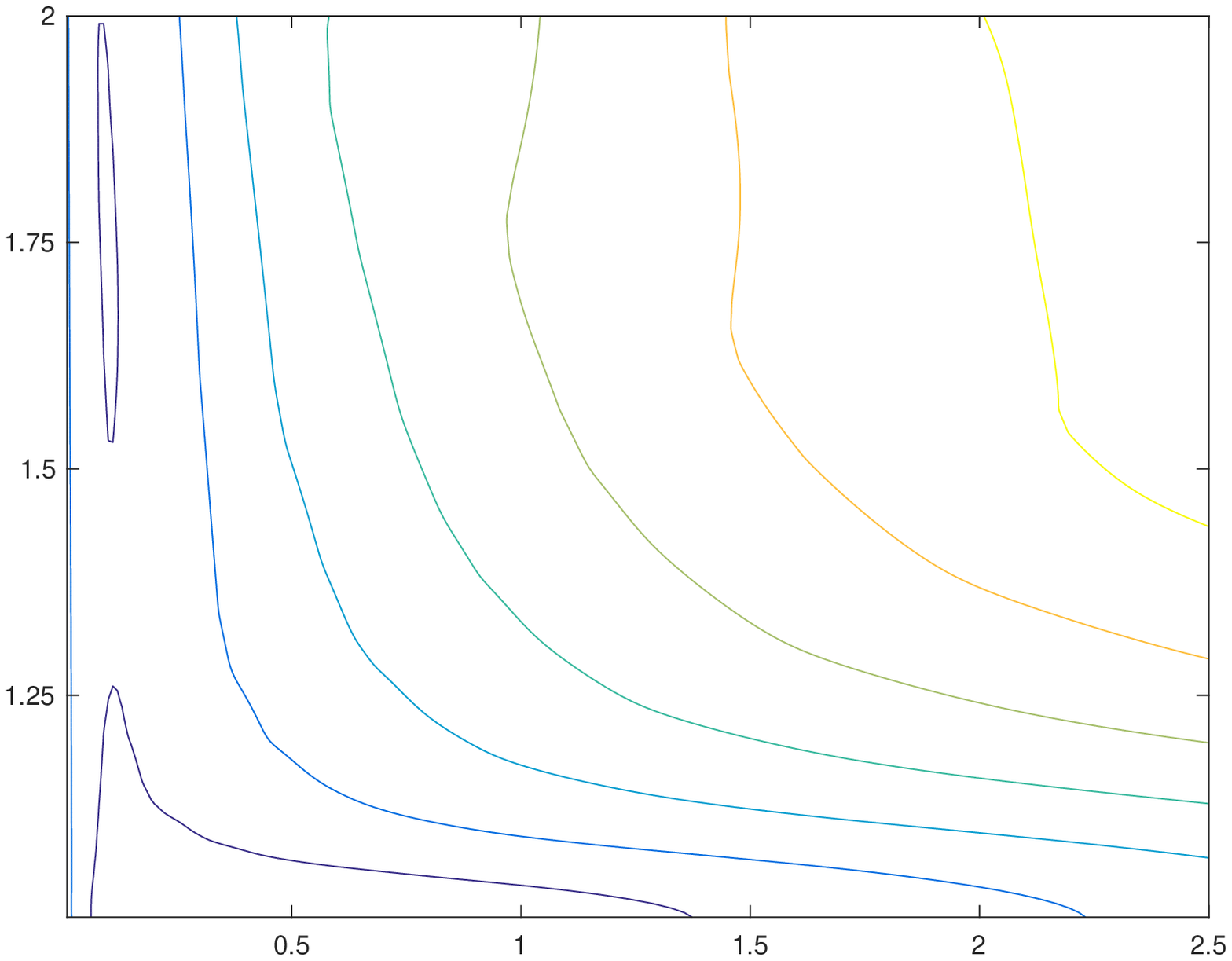}
  \caption{optimal $(\tilde\alpha,\tilde r)=(0.05,1.35)$.}
  \label{fig:flat_cont}
\end{subfigure}
\caption{Numerical landscape (contour map) for the cost function $\mathcal I(\alpha,r)$. Darker colors correspond to smaller values of $\mathcal I(\alpha,r)$}
\label{fig:fig2}
\end{figure}

As we can see, in Figure \ref{fig:flat_cont} the optimal $\tilde r$ seems to lie away from the boundary of the discrete box-constraint \eqref{parameter_domain}, which are the integer values $r=1$ and $r=2$, showing an example in which the optimal signal reconstruction with respect to the $L^2$-distance can be achieved by the fractional order $TGV^r$. We can also see in Figure \ref{fig:fig3} that the optimal denoising results are quite satisfactory. However, it is also possible that the optimal result $\tilde r$ is an integer (in Figure \ref{fig:contour} $\tilde r$ is indeed very close to an integer). For example, for a complete flat signal, i.e., $u_c\equiv1$, then $u_{\alpha,1}=u_c$ for $\alpha$ large enough, provided that the noise $\eta$ satisfies $\int_I \eta(x)\,dx=0$ (zero-average assumption on noise is a reasonable assumption, see \cite{chambolle1997image}). We point out once more that the introduction of fractional $r$-order $TGV^r$ only meant to expand the training choices for the bilevel scheme, but not to provide a superior seminorm to integer order $TGV^k$. The optimal solution $\tilde r \in [1,1/P]$, fractional order or integer order, is completely up to the given data $u_c$ and $u_\eta$.\\\\ 
Although from the numerical landscapes the cost function $\mathcal I(\alpha,r)$ (Figure \ref{fig:fig2}) appears to be almost quasiconvex (see, \cite[Section 3.4]{MR2061575}) in the variable $\alpha$, this is not the case. In the forthcoming paper \cite{de2017numerical} some explicit counterexamples showing that at least for certain piecewise constant signals $\mathcal I\fsp{\alpha, 1}$ is not quasiconvex will be presented. Additionally, Figure \ref{fig:contour} and \ref{fig:flat_cont} both show that $\mathcal I(\alpha,r)$ is not quasiconvex in the $r$ variable. The non-quasiconvexity of the cost function $\mathcal I(\alpha,r)$, implies that the training scheme $(\mathcal R)$ may not have a unique solution, i.e., the global minimizer of $\mathcal I(\alpha,r)$ might be not unique. In particular, the non-quasiconvexity of $\mathcal I(\alpha,r)$ prevents us from using standard gradient descent methods to find a global minimizer (which is the optimal solution we are looking for in \eqref{frac_para_bi}). Therefore, the identification of a reliable numerical scheme for solving the bilevel problem (upper level problem) remains an open question.\\\\
As a final remark, we point out that the development of a numerical scheme to identify the global minimizers of $\mathcal I(\alpha,1)$, $\alpha\in\R^+$ has been undertaken in \cite{de2017numerical}, where the \emph{Bouligand} differentiability and the finite discretization of $\mathcal I(\alpha, 1)$ will be analyzed. \\

\begin{figure}[!h]
\begin{subfigure}{.5\textwidth}
  \centering
  \includegraphics[width=\linewidth]{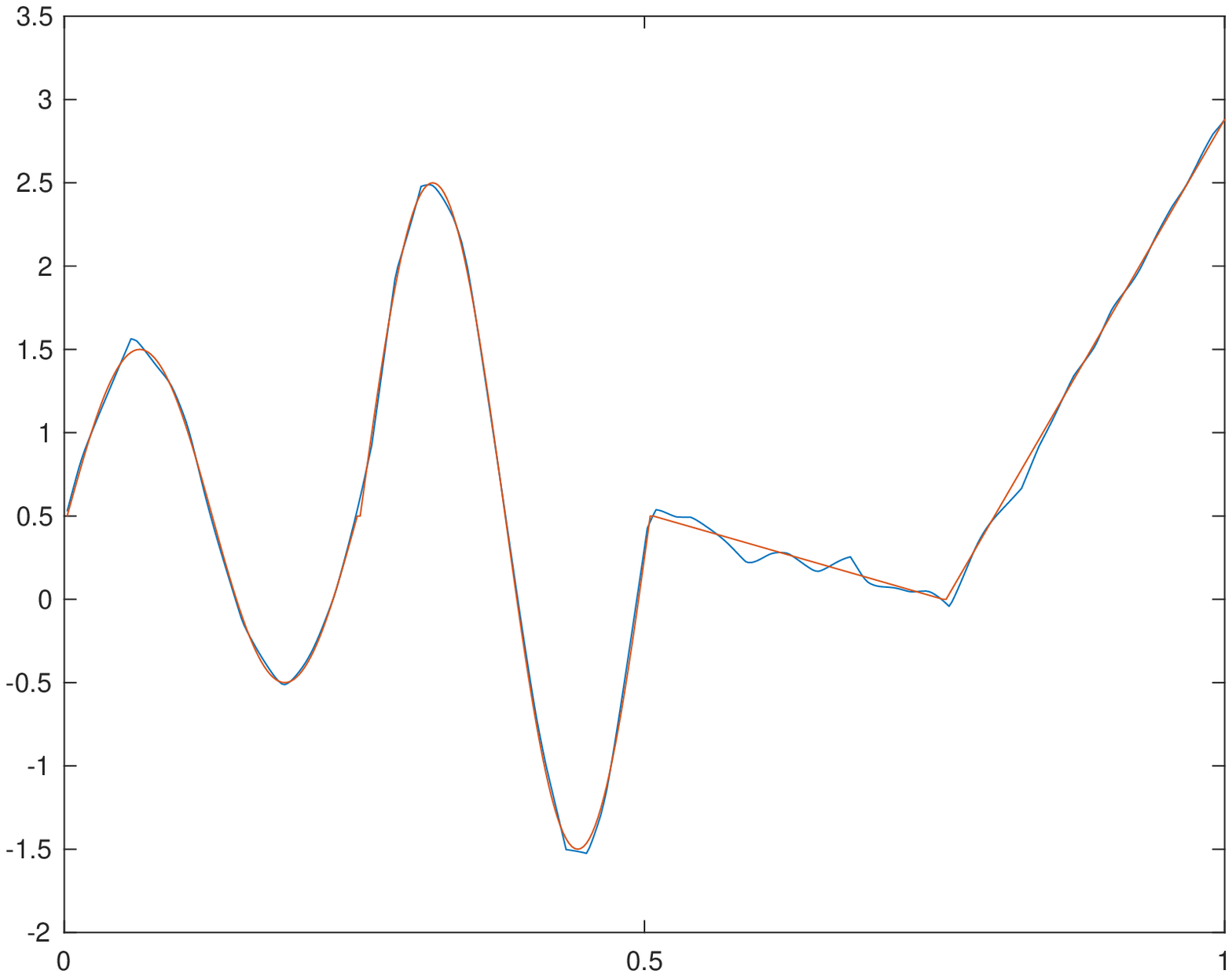}
  \caption{The optimal denoised signal $u_{\tilde \alpha,\tilde r}$ in blue}
  \label{fig:frac_de}
\end{subfigure}%
\begin{subfigure}{.5\textwidth}
  \centering
  \includegraphics[width=\linewidth]{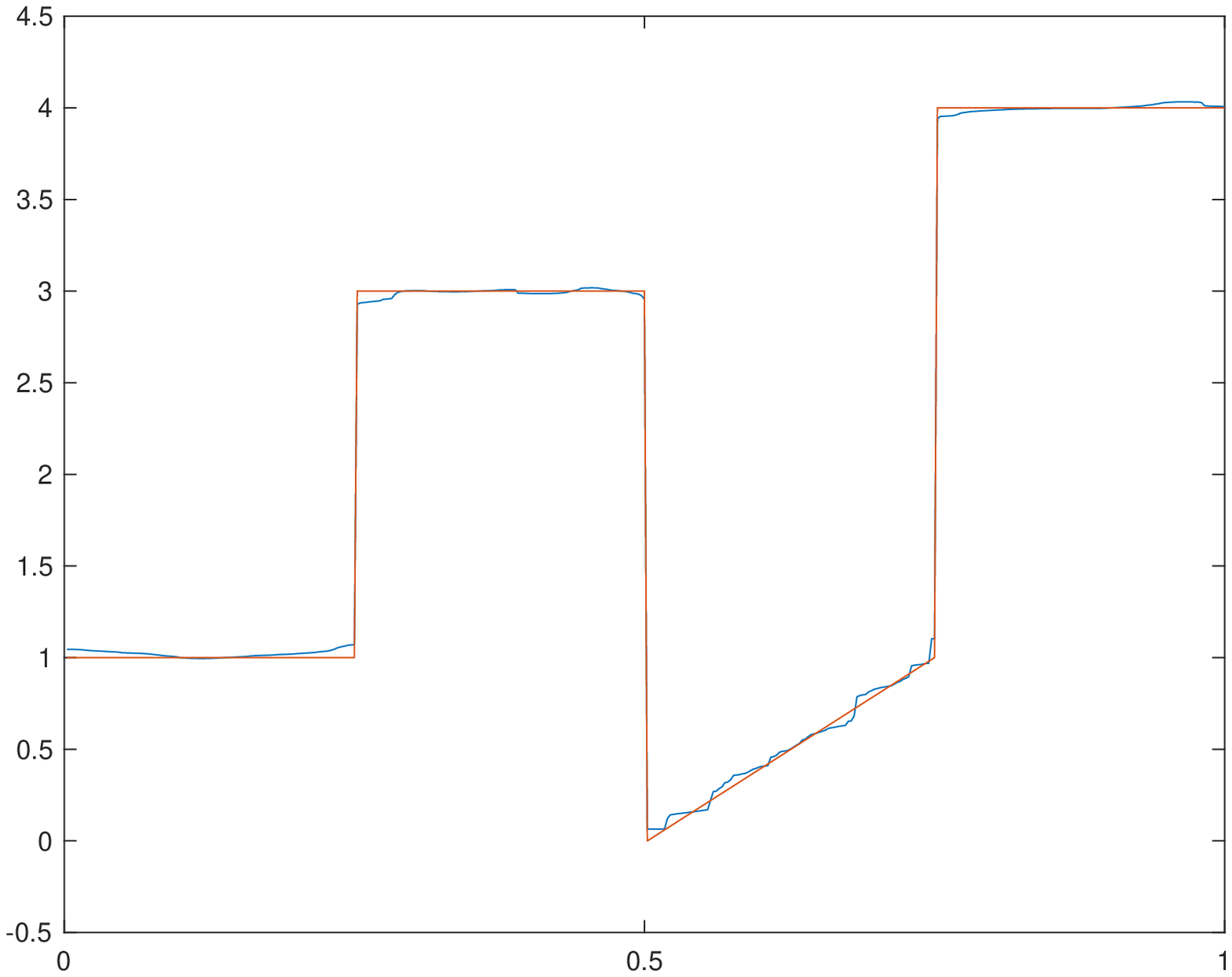}
  \caption{The optimal denoised signal $u_{\tilde \alpha,\tilde r}$ in blue}
  \label{fig:opt_flat}
\end{subfigure}
\caption{The optimal denoised signal $u_{\tilde \alpha,\tilde r}$ in both cases are very close to the given clean signal.}
\label{fig:fig3}
\end{figure}

\section*{Acknowledgements}
P. Liu is partially funded by the National Science Foundation under Grant No. DMS - 1411646. E. Davoli is supported by the Austrian Science Fund (FWF) projects P27052 and F65. The authors wish to thank Irene Fonseca for suggesting this project and for many helpful discussions and comments. The authors also thank  Giovanni Leoni for comments on the subject of Section \ref{prelimy}.

%\newpage

\bibliographystyle{plain}
%\bibliography{Dec_13}{}
%\def\cprime{$'$}

\end{document}